\documentclass[12pt,a4paper]{amsart}
\usepackage{amssymb}
\usepackage{amsfonts}
\usepackage{amsmath}
\usepackage[mathscr]{eucal}
\usepackage{mathrsfs}

\newtheorem{thm}{Theorem}[section]
\newtheorem{prop}[thm]{Proposition}
\newtheorem{lem}[thm]{Lemma}
\newtheorem{cor}[thm]{Corollary}

\numberwithin{equation}{section}

\def\N{{\Bbb N}}
\def\Z{{\Bbb Z}}
\def\Q{{\Bbb Q}}
\def\R{{\Bbb R}}
\def\C{{\Bbb C}}

\def\A{{\Bbb A}}

\def\emp{\varnothing}
\def\fa{{\frak a}}

\def\fh{{\frak h}}

\def\fz{{\frak z}}

\def\sG{\mathsf G}

\def\sP{\mathsf P}

\def\sm{\mathsf m}
\def\sn{\mathsf n}

\def\cO{\frak o}
\def\cL{\mathscr L}
\def\cB{{\mathscr B}}

\def\cH{{\mathscr H}}
\def\cM{{\mathscr M}}

\def\cV{{\mathcal V}}

\def\cF{{\mathscr F}}
\def\cG{{\mathcal G}}

\def\cL{{\mathscr L}}

\def\cW{{\mathcal W}}
\def\cU{{\mathcal U}}
\def\cJ{{\mathcal J}}
\def\Re{{\operatorname {Re}}}
\def\Im{{\operatorname {Im}}}
\def\tr{{\operatorname{tr}}}
\def\nr{{\operatorname{N}}}

\def\Mat{{\operatorname{M}}}

\def\diag{{\operatorname {diag}}}

\def\vol{{\operatorname{vol}}}

\def\leq{\leqslant}
\def\geq{\geqslant}
\def\bsl{\backslash}

\def\z{\zeta}

\def\e{\varepsilon}

\def\cQ{{\mathcal Q}}

\def\d{{\rm{d}}}

\def\fX{{\frak X}}

\def\cD{\mathscr D}

\def\fin{{\rm{\bf {f}}}}
\def\bK{{\bf K}}

\def\b1{{\bold 1}}

\def\cJ{{\mathscr J}}

\def\jj{{j}}
\def\ss{{\mathsf s}}
\newcommand{\sslash}{\mathbin{/\mkern-6mu/}}
\def\bG{{\mathbb G}}

\title{Weighted equidistribution theorem for Siegel modular forms of degree $2$}

\author{Masao Tsuzuki}
\address{Faculty of Science and Technology, Sophia University, Kioi-cho 7-1 Chiyoda-ku Tokyo, 102-8554, Japan}
\email{m-tsuduk@sophia.ac.jp}
\begin{document}
\maketitle

\begin{abstract}
We deduce a weighted equidistribution theorem of the Satake parameters of Sigel cusp forms on ${\bf Sp}_2(\Z)$ with growing even weights. 
\end{abstract}

\section{Introduction} \label{sec:Intro}
Let ${\bf GSp}_2$ be the symplectic similitude group of rank 2, which is a reductive connected algebraic $\Q$-group defined as 
\begin{align*}
{\bf GSp}_2=\{g\in {\bf GL}_4|\,{}^t g\,\left[\begin{smallmatrix} 0 & 1_2 \\ -1_2 & 0 \end{smallmatrix}\right]\,g=\nu(g)\,\left[\begin{smallmatrix} 0 & 1_2 \\ -1_2 & 0 \end{smallmatrix}\right]\,(\exists \nu(g)\in {\bf GL}_1)\},
\end{align*}
whose center ${\bf Z}$ consists of all the scalar matrices in ${\bf GL}_4$. Set ${\bf G}={\bf PGSp}_2:={\bf G}/{\bf Z}$. The identity connected component ${\bf G}(\R)^0$ of real points of ${\bf G}$ transitively acts on the Siegel upper-half space $\fh_2:=\{Z=\left[\begin{smallmatrix} z_1 & z_2 \\ z_2 & z_3 \end{smallmatrix} \right] \in {\bf M}_2(\C)|\,\Im(Z)\gg 0\}$ by 
$$
g.Z=(AZ+B)(CZ+D)^{-1}, \quad g=\left[\begin{smallmatrix} A & B \\ C & D \end{smallmatrix}\right]\in {\bf GSp}_2(\R)^{0},\,Z\in \fh_2.
$$ 
For a positive even integer $l$, let $S_l({\bf Sp}_2(\Z))$ denote the space of Siegel cusp forms of weight $l$, i.e., the set of all those holomorphic bounded functions $\Phi:\fh_2\rightarrow \C$ such that 
\begin{align}
\Phi(\gamma.Z)=\det(CZ+D)^{l}\Phi(Z), \quad \gamma=\left[\begin{smallmatrix} A & B \\ C & D \end{smallmatrix}\right]\in {\bf Sp}_2(\Z).
 \label{Spautomorphy}
\end{align}
The space $S_l({\bf Sp}_2(\Z))$ is a finite dimensional Hilbert space with the inner-product whose associated norm is 
$$
\|\Phi\|^{2}=\int_{{\bf Sp}_2(\Z)\bsl \fh_2}|\Phi(Z)|^2 (\det \Im Z)^{l}\d\mu_{\fh_2}(Z), \quad \Phi \in S_l({\bf Sp}_2(\Z)), 
$$
where 
\begin{align}
\d\mu_{\fh_2}(Z)=(\det \Im Z)^{-3}\prod_{j=1}^{3}2^{-1}|\d z_j\wedge \d\bar z_j|\label{BergmannMetSp}
\end{align}
is the invariant measure on $\fh_2$. Any element $\Phi\in S_l({\bf Sp}_2(\Z))$ is given by its Fourier expansion
$$
\Phi(Z)=\sum_{T\in \cQ^{+}}A_{\Phi}(T)\,e^{2\pi \sqrt{-1}\tr(ZT)}, \quad Z\in \fh_2
$$
with the set of Fourier coefficients $\{A_{\Phi}(T)\}_{T\in \cQ^{+}}$, where $\cQ^{+}$ is the set of positive definite matrices in $\cQ:=\bigl\{T=\left[\begin{smallmatrix} b & a/2 \\ a/2 & c\end{smallmatrix} \right]|\,a,b,c\in \Z\bigr\}$. The latter space $\cQ$ carries an action of the modular group ${\bf SL}_2(\Z)$ given as $\cQ\times {\bf SL}_2(\Z)\ni (T,\delta)\mapsto \delta T{}^t \delta\in \cQ$. From \eqref{Spautomorphy}, the Fourier coefficients $A_{\Phi}(T)$ $(T\in \cQ^{+})$ has the modular invariance $A_{\Phi}(\delta T{}^t\delta)=A_{\Phi}(T)$ $(\delta\in {\bf SL}_2(\Z))$, which allows one to regard $T\mapsto A_{\Phi}(T)$ as a function on the orbit space ${\bf SL}_2(\Z)\bsl \cQ^{+}$. Let $D<0$ be a fundamental discriminant and $\chi$ a character of the ideal class group ${\rm Cl}_{D}$ of the imaginary quadratic field $\Q(\sqrt{D})$. Let $[T]\in {\rm Cl}_{D}$ be the image of $T\in \cQ_{\rm prim}^{+}(D)$ by the natural isomorphism ${\bf SL}_2(\Z)\bsl {\cQ}_{\rm prim}^{+}(D)\cong {\rm Cl}_D$, where 
$$\cQ_{\rm prim}^{+}(D):=\Bigl\{\left[\begin{smallmatrix} b & a/2 \\ a/2 & c\end{smallmatrix} \right]\in \cQ^{+}\Bigm|\,a^2-4bc=D,\,(a,b,c)=1\,\Bigr\}. 
$$
Let $\chi$ be a character of ${\rm Cl}_D$ and $\sigma$ the non trivial element of ${\rm Gal}(\Q(\sqrt{D})/\Q)$. Since $\fa \fa^{\sigma}$ is principal for any invertible ideal $\fa$ of $\Q(\sqrt{D})$, we have that $\chi\chi^{\sigma}$ is trivial; thus $\chi=\chi^{\sigma}$ if and only if $\chi^2={\bf 1}$. Recall that $\chi=\chi^\sigma$ if and only if $\chi$, when viewed as an idele class character of $\Q(\sqrt{D})$, is of the form $\nr_{\Q(\sqrt{D})/\Q}\circ \chi_0$ with some idele class character $\chi_0$ of $\Q$. Following \cite{KST}, let us define
$$
\omega^{\Phi}_{l,D,\chi}:=c_{l,D}\,d_\chi\,\frac{|R(\Phi,D,\chi^{-1})|^2}{\|\Phi\|^2}, \quad \Phi \in S_l({\bf Sp}_2(\Z)), 
$$
where 
$$
R(\Phi,D,\chi):=\sum_{T\in {\bf SL}_2(\Z)\bsl \cQ_{\rm prim}^{+}(D)}A_{\Phi}(T)\,\chi([T])
$$ 
and 
\begin{align*}
d_{\chi}&:=\begin{cases} 1 \quad (\chi^2={\bf 1}), \\
2 \quad (\chi^2\not={\bf 1}),
\end{cases}
\\ 
c_{l,D}&:=
\frac{\sqrt{\pi}}{4}(4\pi)^{3-2l}\Gamma(l-3/2)\Gamma(l-2)
\times \left(\frac{|D|}{4}\right)^{3/2-l}\frac{4}{w_D\,h_D},
\end{align*}
where $w_D$ is the number of roots of unity in $\Q(\sqrt{D})$ and $h_D:=\#{\rm Cl}_D$ is the class number of $\Q(\sqrt{D})$. Let $\cF_l$ be a $\C$-basis of $S_l({\bf Sp}_2(\Z))$ consisting of joint-eigenfucntions of all the Hecke operators. In the work \cite{KST}, Kowalski-Saha-Tsimerman investigated the quantity $\omega_{l,D,\chi}^{\Phi}$ from a statistical point of view, including the asymptotic behavior of the average of spinor $L$-values $L_{\fin}(s,\pi_\Phi)$ for $s$ on the convergent range of the Euler product taken over the ensemble $\{\omega_{l,D,\chi}^{\Phi}|\,\Phi\in \cF_l\}$ with growing $l$. Later, the asymptotic formula for the central spinor $L$-values is proved by Blomer in \cite{Blomer}, where even a second moment formula is erabolated by a deep analysis of diagonal and off-diagonal cancellation of terms from the Petersson formula for Siegel modular forms. In our previous paper \cite{Tsud2019}, based on a different technique involving the archimedean Shintani functions and Liu's computation of local Bessel priods for spherical functions, we extend the (first moment) asymptotic formula for central standard $L$-values of cusp forms on ${\rm SO}(2,m)$ $(m\geq 3)$ in a general setting. In this paper, we examine the case when $m=3$ in detail. 
\subsection{Description of results}
To state the main result, we need additional notation. For $\Phi \in \cF_l$, let $\pi_\Phi$ be the automorphic representation of ${\bf G}(\A)$ generated by the function $\tilde \Phi$ on the adeles ${\bf G}(\A)$ well-defined by the relation $\tilde \Phi(\gamma g_\infty u_\fin )=\det(\sqrt{-1}C+D)^{-l}\Phi((A\sqrt{-1}+B)(C\sqrt{-1}+D)^{-1})$ for $\gamma \in {\bf G}(\Z)$, $g_\infty=\left[\begin{smallmatrix} A & B \\ C & D \end{smallmatrix}\right]\in {\bf G}(\R)^{0}$ and $u_\fin \in {\bf G}(\widehat \Z)$. By \cite[Corollary 3.3]{NPS}, $\pi_\Phi$ is irreducible and cuspidal; as such it can be decomposed as the restricted tensor product $\pi_\Phi\cong \bigotimes_{p\leq \infty} \pi_{\Phi,p}$ of irreducible smooth representations $\pi_{\Phi,p}$ of ${\bf G}(\Q_p)$ for $p<\infty$ and $\pi_{\Phi,\infty}$ a holomorphic discrete series representation of ${\bf G}(\R)$ of scalar weight $l$. Let ${\bf B}$ be a Borel subgroup consisting of all matrices in ${\bf G}$ of the form 
\begin{align}
\left[\begin{smallmatrix} A & 0 \\ 0 & \lambda {}^{t}A^{-1} \end{smallmatrix} \right]\,\left[\begin{smallmatrix} 1_2 & B \\ 0 & 1_2 \end{smallmatrix} \right] \quad ((\lambda,A)\in {\bf GL}_1 \times{\bf GL}_2, \quad B={}^t B)
 \label{SiegelParab}
\end{align}
with $A$ being an upper-triangular matrix of degree $2$. Let ${\bf U}$ denote the unipotent radical of ${\bf B}$, which consists of all the elements \eqref{SiegelParab} such that $A$ is an upper-triangular unipotent matrix. For a prime number $p$, set 
$$
\fX_p:=(\C/2\pi \sqrt{-1}(\log p)^{-1}\Z)^2
$$
and $W(C_2)$ the $C_2$-Weyl group which, as an automorphism group of $\fX_p$, is generated by the two elements $s_1,s_2$ given as $s_1(\nu_1,\nu_2)=(\nu_2,\nu_1)$ and $s_2(\nu_1,\nu_2)=(\nu_1,-\nu_2)$. For $\nu=(\nu_1,\nu_2)\in \fX_p$, let $I_p(\nu)={\rm Ind}_{{\bf B}(\Q_p)}^{{\bf G}(\Q_p)}(\chi_\nu)$ denote the parabolically induced representation of ${\bf G}(\Q_p)$ from a quasi-character $\chi_\nu$ of ${\bf B}(\Q_p)$ given as
\begin{align}
\chi_\nu(\diag(t_1,t_2,\lambda t_1^{-1},\lambda t_2^{-1})n)=|t_1|_p^{-\nu_{1}+\nu_2}|t_{2}|_p^{-\nu_1-\nu_{2}}|\lambda|_p^{\nu_1}, (t_1,t_2,\lambda)\in (\Q_p^{\times})^3,\,n\in {\bf U}(\Q_p).
 \label{GspBorelchar}
\end{align}
It is known that $I_p(\nu)$ admits a unique ${\bf G}(\Z_p)$-spherical constituent to be denoted by $\pi_p^{{\rm ur}}(\nu)$. Note that $\pi_p^{{\rm ur}}(w\nu)\cong \pi_p^{{\rm ur}}(\nu)$ for all $\nu \in \fX_p$ and $w\in W(C_2)$. The local spinor $L$-factor attached to $\pi_{p}^{\rm ur}(\nu)$ is defined as 
$$
L(s,\pi^{\rm ur}(\nu))=\prod_{j=1}^{2} (1-\alpha_{j} p^{-s})^{-1}
(1-\alpha_j^{-1}p^{-s})^{-1}
$$
with $\alpha_j=p^{-\nu_j}\,(j=1,2)$. Let $\nu_p(\Phi)=(\nu_{1,p},\nu_{2,p})\in \fX_p/W(C_2)$ be the unique point such that $\pi_{\Phi,p} \cong \pi_p^{\rm ur}(\nu_p(\Phi))$.
The spinor $L$-function $L_\fin(s,\pi_\Phi)$ of $\pi_\Phi$ and its completion $L(s,\pi_\Phi)$ are originally defined as the degree $4$ Euler product 
\begin{align*}
L(s, \pi_\Phi)&:= \Gamma_{\C}(s+1/2)\Gamma_{\C}(s+l-3/2)\times L_\fin(s,\pi_\Phi), \\
L_\fin(s,\pi_\Phi)&:=\prod_{p<\infty}L(s,\pi_p^{\rm ur}(\nu_p(\Phi))), \quad \Re s\gg 0,
\end{align*}
where $\Gamma_{\C}(s):=2(2\pi)^{-s}\Gamma(s)$. In this paper, we use the symbol $\fin $ to denote the set of all the prime numbers, or as a subscript to indicate that the object is related to the set of finite adeles. It is known that $L(s,\pi_\Phi)$ has a meromorphic continuation to $\C$ with the functional equation $L(1-s,\pi_\Phi)=L(s,\pi_\Phi)$ admitting possible poles at $s=3/2,-1/2$ (\cite{And1} and \cite{And2}). It should be also recalled that these poles are at most simple and they occur if and only if $\Phi$ is the Saito-Kurokawa lifting from an elliptic cusp form on ${\bf SL}_2(\Z)$ (\cite{Oda}, \cite{PS1983}).

Let ${\mathcal {AI}}(\chi)\cong \bigotimes_{p\leq \infty}{\mathcal {AI}}(\chi)_p$ be the automorphic induction from an idele class character $\chi$ of $\Q(\sqrt{D})$, which is an isobaric automorphic representation of ${\bf GL}_2(\A)$; it is not cuspdal if and only if $\chi=\chi_0 \circ \nr_{\Q(\sqrt{D})/\Q}$ with some Hecke character $\chi_0$ of $\Q$ in which case ${\mathcal {AI}}(\chi)=\chi_0 \boxplus \chi_0\eta_{D}$, where $\eta_D$ is the quadratic idele class character of $\Q$ corresponding to $\Q(\sqrt{D})$ by class field theory. Let $L_\fin(s,{\mathcal {AI}}(\chi))$ be the Hecke $L$-function (degree $2$) of the automorphic representation ${\mathcal {AI}}(\chi)$. By transcribing \cite[Theorem 1]{Tsud2019} in the language of Siegel modular forms, we have the following result.   
\begin{thm} \label{thm0}
Let $D<0$ be a fundamental discriminant and $\chi$ a character of ${\rm Cl}_{D}$. Then there exists a constant $C=C(D)>1$ (independent of $\chi$) such that as $l\in 2\N$ grows to infinity, \begin{align*}
\sum_{\Phi \in \cF_l} L_{\fin}(1/2,\pi_\Phi)\,\omega_{l,D,\chi^{-1}}^{\Phi}
=2\,P(l,D,\chi)+O(C^{-l})
\end{align*}
with  
\begin{align}
P(l,D,\chi)
&=\begin{cases}
L_\fin(1,\eta_D)\,(\psi(l-1)-\log(4\pi^2))+L'_\fin(1,\eta_D), \quad &(\chi={\bf 1}), \\ 
L_\fin(1,{\mathcal {AI}}(\chi)), \quad &(\chi\not={\bf 1}),
\end{cases}
\notag
\end{align}
where $\psi(s)=\Gamma'(s)/\Gamma(s)$ is the di-gamma function. 
\end{thm}
After recalling a basic setting for orthogonal groups in \S~\ref{sec:Othogonal}, we state the corresponding asymptotic formula for the orthogonal group in Corollary~\ref{cor:Oasympt}, from which Theorem~\ref{thm0} is easily deduced by the materials collected in \S~\ref{sect:Bookkeep}. Since $\psi(l-1)=\log l+O(l^{-1})$ as is well-known, Theorem~\ref{thm0} when specialized to the case $D=-4$ and $\chi={\bf 1}$ recovers the asymptotic formula stated in \cite[Theorem 1]{Blomer}. Note that our asymptotic formula has a much stronger error term $O(C^{-l})$ than $O(l^{-1})$ ({\it cf}. \cite[(1.8)]{Blomer}).

For each prime number $p$, we fix a Haar measure $\d g_p$ on ${\bf G}(\Q_p)$ such that $\vol({\bf G}(\Z_p))=1$. Let $\cH({\bf G}(\Q_p)\sslash {\bf G}(\Z_p))$ be the spherical Hecke algebra of ${\bf G}(\Q_p)$. For any function $\phi \in \cH({\bf G}(\Q_p)\sslash {\bf G}(\Z_p))$, let $\hat \phi:\fX_p\rightarrow \C$ denote the spherical Fourier transform of $\phi$, i.e, $\hat\phi(\nu)$ is the eigenvalue of $\pi_{p}^{\rm ur}(\nu)(\phi)=\int_{{\bf G}(\Q_p)}\phi(g_p)\pi_{p}^{\rm ur}(g_p)\d g_p$ on the ${\bf G}(\Z_p)$-fixed vectors of $\pi_p^{\rm ur}(\nu)$. Let $\d \mu_{p}^{\rm Pl}$ be the spherical Plancherel measure corresponding to $\d g_p$, i.e., a non-negative Radon measure on $\fX_p$ supported on the tempered locus $\fX_p^{0}=(\sqrt{-1}\R/2\pi \sqrt{-1}(\log p)^{-1}\Z)^2$ which fits in the inversion formula: $$
\int_{\fX_p^{0}}\hat \phi(\nu)\,\d\mu_{p}^{\rm Pl}(\nu)=\phi(1_4), \quad \phi \in\cH({\bf G}(\Q_p)\sslash {\bf G}(\Z_p)).
$$
Let $S$ be a finite set of prime numbers. For any $\alpha=\otimes_{p\in S}\alpha_p$ continuous function on $\fX_S=\prod_{p\in S}(\C/2\pi \sqrt{-1}(\log p)^{-1}\Z)^2$, define 
\begin{align*}
{\bf \Lambda}_S^{\chi}(\alpha):=\prod_{p\in S}\, \frac{\zeta_p(2)\zeta_p(4)}{
\zeta_p(1)L(1,{\mathcal {AI}}(\chi)_p)} 
\int_{\fX_p^{0}/W(C_2)} \frac{L\left(\frac{1}{2},\pi_p^{\rm ur}(\nu) \times {\mathcal {AI}}(\chi)_p\right)L\left(\frac{1}{2},\pi_p^{\rm ur}(\nu)\right)}{L(1,\pi_p^{\rm ur}(\nu),{\rm Ad})}\,\d \mu_{p}^{{\rm Pl}}(\nu)
\end{align*}
and $\mu_{S}^{\rm Pl}=\bigotimes_{p\in S}\mu_p^{\rm Pl}$, where $L(s,\pi_p^{\rm ur}(\nu)\times {\mathcal {AI}}(\chi)_p)$ is the local $p$-factor of the ${\bf GSp}_2 \times {\bf GL}_2$ convolution $L$-function (degree $8$) and $L(s,\pi_p^{\rm ur}(\nu),{\rm Ad})$ is the local $p$-factor of the adjoint $L$-function of ${\bf GSp}_2$ (degree $10$). Let $\fX_p^{0+}$ denote the set of $\nu \in \fX_p$ such that $\pi_p^{\rm ur}(\nu)$ is unitarizable. Note that $\fX_p^{0+}$ is a relatively compact subset of $\fX_p$ and $\fX_p^{0}\subset \fX_p^{0+}$. Since $\pi_\Phi$ with $\Phi\in \cF_l$ is a subrepresentation of $L^2({\bf G}(\Q)\bsl {\bf G}(\A))$, the local components $\pi_{\Phi,p}$ are unitarizable, which implies $\nu_p(\Phi)\in \fX_p^{0+}$ for all $p<\infty$. For a set $S$ of primes, let $\nu_S(\Phi)$ denote the element $\{\nu_p(\Phi)\}_{p\in S}$ of $\fX_S^{0+}:=\prod_{p\in S}\fX_p^{0+}$. Now we can state our main theorem as follows. 

\begin{thm} \label{MAINTHM}
Let $D<0$ be a fundamental discriminant and $\chi$ a character of ${\rm Cl}_D$. For $l\in 2\N$, let $\cF_{l}$ be a Hecke eigen basis of $S_l({\bf Sp}_2(\Z))$ and $\cF^{\#}_l$ the set of $\Phi\in \cF_l$ which is a Saito-Kurokawa lifting from elliptic cusp forms on ${\bf SL}_2(\Z)$. Set $\cF^{\flat}_l=\cF_l-\cF_l^{\#}$. Let $S$ be a finite set of odd prime numbers such that $p\not\in S$ for all prime $p|D$. Then for any $\alpha \in C(\fX_S^{0+}/W_S)$, as $l\in 2\N$ grows to infinity,   
\begin{align*}
&\frac{1}{(\log l)^{\delta(\chi={\bf 1})}}\,\sum_{\Phi \in \cF_l^{\flat}}\alpha(\nu_{S}(\Phi))\,{L_\fin(1/2,\pi_\Phi)}\,\omega^{\Phi}_{l,D,\chi^{-1}}\,
\rightarrow 2 {\bf \Lambda}_S^\chi (\alpha)\,\begin{cases} 
L_\fin(1,\eta_D), \quad (\chi={\bf 1}), \\ 
L_{\fin}(1,{\mathcal {AI}}(\chi)), \quad (\chi\not={\bf 1}),
\end{cases} \\
&\frac{1}{(\log l)^{\delta(\chi={\bf 1})}}\,\sum_{\Phi \in \cF_l^{\#}}\alpha(\nu_{S}(\Phi))\,{L_\fin(1/2,\pi_\Phi)}\,\omega^{\Phi}_{l,D,\chi^{-1}}\,
\rightarrow 0.
\end{align*}
\end{thm}
We note that the proof of this theorem requires the non-negativity $L_{\fin}(1/2,\pi_\Phi)\geq 0$ ($\forall \Phi\in \cF_l^{\flat}$), which is known (\cite[Theorem 5.2.4]{PSS}, \cite{Lapid}, \cite{Weissauer}). 
\begin{cor}\label{MAINTHM2} Let $D<0$ be a fundamental discriminant and $S$ a finite set of add prime numbers such that $p\in S$ is relatively prime to $D$. Let $\chi$ be a character of ${\rm Cl}_{D}$. Given a Riemann integrable subset $U$ of $\fX_S^{0}/W_S$ such that $\mu_{S}^{\rm Pl}(U)>0$, there exists $l_0\in \N$ with the following property: for any even integer $l>l_0$ there exists $\Phi \in \cF_l^{\flat}$ such that 
\begin{itemize}
\item[(i)] $L_{\fin}(1/2,\pi_\Phi)>0$, 
\item[(ii)] $R(\Phi,D,\chi)\not=0$,
\item[(iii)] $\nu_{S}(\Phi)\in U$. 
\end{itemize}
\end{cor}

At this point, we should recall a conjecture by Dickson-Pitale-Saha-Schmidt (\cite{DPSS}), which is a generalization of B\"{o}chere's conjeture(\cite{Boechere}) and is deduced from a version of the refined Gan-Gross-Prasad conjecture posed by Y.Liu (\cite{Liu}): 

\medskip
\noindent
{\bf Conjecture (\cite[Conjecture 1.3]{DPSS})} : Let $l>2$ be an even integer and $\Phi\in S_l({\bf Sp}_2(\Z))$ is a joint eigenfunction of all the Hecke operators. Suppose that $\Phi$ is not the Saito-Kurokawa lifting from an elliptic cusp form on ${\bf SL}_2(\Z)$. Then for any fundamental discriminant $D<0$ and for any character $\chi$ of ${\rm Cl}_{D}$, 
\begin{align}
\frac{|R(\Phi,D,\chi^{-1})|^2}{\|\Phi\|^2 }
=\frac{2^{4l-4}\pi^{2l+1}}{(2l-2)!}w_D^2|D|^{l-1}\frac{L_\fin(1/2,\pi_\Phi \times {\mathcal {AI}}(\chi))}{L_\fin(1,\pi_\Phi,{\rm Ad})}.
 \label{GGPPeriod}
\end{align}

\smallskip
\noindent
Note that the analytical prperties of $L$-functions appering in the formula are fully studied in \cite{PSS}: in particular, it is proved that both the degree $8$ $L$-function $L(s,\pi_{\Phi}\times {\mathcal {AI}}(\chi))$ and the degree $10$ $L$-function $L(s,\pi_{\Phi};{\rm Ad})$ are entire and that $L_\fin(1,\pi_\Phi,{\rm Ad})\not=0$ (\cite[Theorem 4.1.1, Theorem 5.2.1]{PSS}). Conditionally upon this conjecture, given $U$ and $\chi$ as above, Corollary~\ref{MAINTHM2} yields an infinite family of Siegel modular forms $\Phi \in S_l({\bf Sp}_2(\Z))$ with growing weights such that 
\begin{align*}
\text{$L_{\fin}(1/2,\pi_\Phi)\,L_\fin(1/2,\pi_\Phi \times {\mathcal {AI}}(\chi))\not=0$ and $\nu_{S}(\Phi)\in U$.}
\end{align*}   

The validity of the conjecture when $\chi$ is trivial is proved by Furusawa-Morimoto (\cite{FurusawaMorimoto}): 
\begin{thm}$(${\rm Furusawa-Morimoto} \cite[Theorem 2]{FurusawaMorimoto}$)$
 Let $\Phi\in S_{l}({\bf Sp}_2(\Z))$ with an evem $l>2$ is a joint eigenfunction of all the Hecke operators on ${\rm Sp}_2(\Z)$. Suppose that $\Phi$ is not a Saito-Kurokawa lift. For any negative fundamental discriminant $D$, when $\chi$ is the trivial character of ${\rm Cl}_{D}$, the equality \eqref{GGPPeriod} is true.  \end{thm}

Invoking this, we have the following result unconditionally. 
\begin{cor}
Let $D<0$ be a fundamental discriminant and $S$ a finite set of add prime numbers such that $p\in S$ is relatively prime to $D$. Let $\chi$ be a character of ${\rm Cl}_{D}$. Given a Riemann integrable subset $U$ of $\fX_S^{0}/W_S$ such that $\mu_{S}^{\rm Pl}(U)>0$, there exists $l_0\in \N$ with the following property: for any even integer $l>l_0$ there exists $\Phi \in \cF_l^{\flat}$ such that \begin{itemize}
\item[(i)] $L_{\fin}(1/2,\pi_\Phi)\,L_\fin(1/2,\pi_\Phi \times \eta_D)>0$, 
\item[(ii)] $\nu_{S}(\Phi)\in U$. 
\end{itemize}
\end{cor}
We should remark that when $S=\emp$, this corollary also follows from \cite[Theorem 3.15]{DPSS}.

\section{Preliminaries}
In this section we recall well-known facts on automorphic forms on the anisotropic orthogonal group of degree $2$ in the framework of \cite{MS98}.

\subsection{A general setting}
Let $(V_1,Q_1)$ be a non-degenerate quadratic space over $\Q$ such that $\dim(V_1)=m$ and $V_1$ is isotropic. Let $\cL_1$ be a maximal integral lattice in $(V_1,Q_1)$, i.e., $2^{-1}Q_1(\cL_1)\subset \Z$ and if $\cM$ is a $\Z$-lattice such that $2^{-1}Q_1(\cM)\subset \Z$ and $\cL_1\subset \cM$ then $\cM=\cL_1$. The associated bi-linear form $Q_1(X,Y)=2^{-1}(Q_1(X+Y)-Q_1(X)-Q_1(Y))$ ($X,Y\in V_1$) on $V_1$ takes integral values on $\cL_1\times \cL_1$. Let $\cL_1^*:=\{X\in V_1|\,Q_1(X,\cL_1)\subset \Z\}$ be the dual lattice of $\cL_1$, and $\xi\in \cL_1^{*}$ a reduced vector, i.e., $\xi$ is primitive in $\cL_1^{*}$ and the lattice $\cL_1^{\xi}:=\cL_1\cap V_1^{\xi}$ is maximal integral in $(V_1^\xi,Q^\xi)$, where $V_1^\xi:=\{X\in V_1|Q_1(X,\xi)=0\}$ is the orthogonal complement of $\Q\xi$ and $Q_1^\xi=Q_1|V_1^\xi$. Set
\begin{align*}
\sG_1={\bf O}(Q), \quad \sG_1^{\xi}={\rm Stab}_{\sG_1}(\xi)\cong {\bf O}(Q_1^\xi). 
\end{align*}
For each prime number $p$, define
\begin{align*}
\bK_{1,p}&=\{g\in \sG_1(\Q_p)|\,g\cL_{1,p}=\cL_{1,p}\}, \quad \bK_{1,p}^{*}:=\{g\in \bK_{1,p}|\,(g-1)\,\cL_{1,p}^{*}\subset \cL_{1,p}\}, \\ 
\bK_{1,p}^{\xi}&=\{h\in \sG_1^{\xi}(\Q_p)|\,h\cL_{1,p}^{\xi}=\cL_{1,p}^{\xi}\}, \quad 
\bK_{1,p}^{\xi *}=\{h\in \bK_{1,p}^{\xi}|\,(h-1)\,\cL_{1,p}^{\xi *}\subset \cL_{1,p}^{\xi}\}, 
\end{align*}
where $\cL_{1}^{\xi*}$ is the dual lattice of $\cL_1^{\xi}$ in $V_1^\xi(\Q)$. From \cite[]{MS98}, we have
\begin{align}
\bK_{1,p}^{*}\cap \sG_1^{\xi}(\Q_p)=\bK_{1,p}^{\xi*} \quad (p<\infty).
\label{k+k+}
\end{align}
We suppose $\bK_{1,p}=\bK_{1,p}^{*}$ for all $p<\infty$ from now on, and set $\bK_{1,\fin}=\prod_{p<\infty}\bK_{1,p}$ etc. From \cite[Theorem 5.1]{PlRap}, there exists a finite subset $\{u_j\}_{j=1}^{t} \subset \sG_1(\A_\fin)$ with the disjoint decomposition:
\begin{align}
\sG_1(\A)=\bigcup_{j=1}^{t} \sG_1(\Q)u_{j}\sG_1(\R)\bK_{1,\fin},
 \label{Disj}
\end{align}
where $t$ is the class number of $\sG_1$. For $u=(u_p)_{p<\infty} \in \sG_1(\A_\fin)$, define
\begin{align*}
\cL_1(u)&:=V_1(\Q)\cap (V_1(\R)\,\prod_{p<\infty}u_p\,\cL_{1,p}), \\
\Gamma_{Q_1}(u)&:=\sG_1(\Q)\cap (\sG_1(\R)\prod_{p<\infty} u_p\bK_{1,p}u_p^{-1}). \end{align*}
Let $\cL_1(u)^{*}$ be the dual lattice of $\cL_1(u)\subset V_1(\Q)$. For $\Delta\in \Q$, set 
$$
\cL_1(u)^{*}_{{\rm prim}, [\Delta]}:=\{\eta \in \cL_1(u)^{*}_{\rm prim}|\,Q_1(\eta)=\Delta\}.
$$
\begin{prop}\label{Prop1} Set $\Delta=Q_1(\xi)$. 
There exists a bijective map  
\begin{align*}
\bar \j: \sG_1^{\xi}(\Q)\bsl \sG_1^{\xi}(\A_\fin)/\bK_{1,\fin}^{\xi *} 
{\rightarrow}\bigsqcup_{j=1}^{t}(\Gamma_{Q_1}(u_j)\bsl 
\cL_1(u_j)^{*}_{{\rm prim}, [\Delta]})
\end{align*}
such that for any $\bar h\in \sG_1^\xi(\Q)\bsl \sG_1^\xi(\A_\fin)/\bK_{1,\fin}^{\xi *}$ represented by $h\in \sG_1^\xi(\A_\fin)$ and a representative $\eta\in \cL_1(u_j)^{*}_{\rm prim}$ of $\bar \j(\bar h)\in \Gamma_{Q_1}(u_j)\bsl \cL_1(u_j)^{*}_{\rm{prim},[\Delta]}$,  
\begin{align}
\#(\sG_1^{\xi}(\Q)\cap h\bK_{1,\fin}^{\xi *}h^{-1})=\#(\Gamma_{Q_1}(u_j)_{\eta}),
 \label{P2}
\end{align}
where $\Gamma_{Q_1}(u_j)_{\eta}=\{\gamma\in \Gamma_{Q_1}(u_j)|\gamma\eta=\eta\}$. 
\end{prop}
\begin{proof} 
Let us define a map 
$$\j:\sG_1^{\xi}(\A_\fin) \rightarrow X:=\bigsqcup_{j=1}^{t}(\Gamma_{Q_1}(u_j)\bsl \cL_1(u_j)^{*}_{{\rm prim}, [\Delta]})
$$
as follows: Let $h\in \sG_1^\xi(\A_\fin)$ and write it as 
\begin{align}
h&=\gamma u_j g_\infty g_\fin \quad 
\text{with $\gamma\in \sG_1(\Q)$, $1\leq j\leq t$, $g_\infty \in \sG_1(\R)$ and $g_\fin\in \bK_{1,\fin}$.}
 \label{P1}
\end{align}
Since \eqref{Disj} is a disjoint union, $j$ is uniquely determined by $h$. Then the vector $\gamma^{-1}\xi \in V$ belongs to the lattice $\cL_1(u_j)^{*}$ and its $\Gamma_{Q_1}(u_j)$-orbit does not depend on the decomposition \eqref{P1}. Indeed, by looking at the finite component of \eqref{P1}, we have $h=\gamma u_j g_\fin$, or equivalently $\gamma^{-1}=u_j g_\fin h^{-1}$. Hence $\gamma^{-1}\xi=u_{j}g_\fin\,\xi$, which implies $(\gamma^{-1}\xi)_{p}=u_{j,p}g_p\,\xi_p\in u_{j,p}g_p\cL_{1,p}^{*}=u_{j,p}\cL_{1,p}^{*}=(\cL_1(u)^{*})_{p}$ for all $p<\infty$. Thus $\gamma^{-1}\xi \in \cL_1(u)^{*}$. If $h=\gamma'u_j g_\infty' g_\fin'$ be another decomposition like \eqref{P1}. Then $\gamma u_j g_\infty g_\fin=\gamma' u_j g_\infty' g_\fin'$ yields the relation $\gamma_\fin u_j g_\fin=\gamma_\fin' u_j g_\fin'$, or equivalently $\gamma_\fin^{-1}\gamma_\fin'=u_j(g_\fin (g_\fin')^{-1})u_j^{-1}$, which implies $\gamma^{-1}\gamma'\in \sG_1(\Q)\cap (\sG_1(\R)\,u_j\bK_{1,\fin}^{*}u_j^{-1})=\Gamma_{Q_1}(u_j)$. Thus $\gamma^{-1}\xi=\delta\,(\gamma')^{-1}\xi$ with some $\delta \in \Gamma_{Q_1}(u_j)$ as desired. 

Therefore, we have a well-defined map $\j:\sG_1(\A_\fin)\rightarrow X$ such that $$
\j(h)=\Gamma_{Q_1}(u_j)\,\gamma^{-1}\xi
$$
for any $h\in \sG_1(\A_\fin)$ with the decomposition \eqref{P1}. From this it is evident that $\j(\delta h k)=\j(h)$ for all $\delta\in \sG_1^{\xi}(\Q)$ and $k\in \bK_{1,\fin}^{*}\cap \sG_1^{\xi}(\A_\fin)$. By \cite[Proposition 2.3]{MS98}, we have $\bK_{1,\fin}^{*}\cap \sG_1^\xi(\A_\fin)=\bK_{1,\fin}^{\xi *}$. Hence by passing to the quotient, the map $\j$ induces a map 
$$
\bar \j: \sG^{\xi}(\Q)\bsl \sG^\xi(\A_\fin)/\bK_\fin^{\xi *} \rightarrow X.
$$
To confirm the injectivity of $\bar\j$, take $h,h'\in \sG_1^\xi(\A_\fin)$ with $\bar \j(h)=\bar \j(h')$. Let $h'=\gamma' u_i \gamma_\infty' g_\fin'$ be the decomposition of $h'$ like \eqref{P1}. Since $j$ is determined by $\bar \j(h)$ from the relation $\bar \j(h)\in \Gamma_{Q_1}(u_j)\bsl \cL_1(u_j)^{*}_{[\Delta]}$, we have $i=j$. Then the relation $\bar \j(h)=\bar \j(h')$ implies $\gamma^{-1}\xi=\delta\,(\gamma')^{-1}\xi$ with some $\delta\in \Gamma_{Q_1}(u_j)$. Hence $\beta:=\gamma'\delta^{-1}\gamma^{-1}\in \sG_1^{\xi}(\Q)$. Since $\gamma^{-1}\xi=u_j g_\fin \xi$ and $(\gamma')^{-1}\xi=u_j g_\fin'\xi$ in $V_1(\A_\fin)$, we also have $u_jg_\fin \xi=\delta_\fin u_j g_\fin ' \xi$, from which the element $g_{\fin}^{-1}u_j^{-1}\delta_\fin u_j g_\fin'$ is seen to belong to $\sG_1^{\xi}(\A_\fin)$. The last element also belongs to $\bK_\fin^*$ due to $\delta \in \Gamma_{Q_1}(u_j)$. Hence $\kappa^{-1}:=g_{\fin}^{-1}u_j^{-1}\delta_\fin u_j g_\fin'\in \sG_1^\xi(\A_\fin)\cap \bK_{1,\fin}^*=\bK_{1,\fin}^{\xi *}$. Using this, we have
\begin{align*}
h=\gamma_\fin u_j g_\fin
=\beta_\fin\gamma'_\fin (\delta_\fin ^{-1}u_j g_\fin)
=\beta_\fin\gamma_\fin (u_j g_\fin'\kappa) 
=\beta_\fin h' \kappa.
\end{align*}
This shows $h$ and $h'$ determines the same double coset in $\sG_1^{\xi}(\Q)\bsl \sG_1^\xi(\A_\fin)/\bK_{1,\fin}^{\xi *}$. 

Let us show the surjectivity of $\bar \j$; let $\eta \in \cL_1(u_j)^{*}_{{\rm prim}, [\Delta]}$ with $1\leq j\leq t$ and find $h\in \sG_1^{\xi}(\A_\fin)$ such that $\j(h)=\Gamma_{Q_1}(u_j)\eta$. Since $Q_1[\xi]=Q_1[\eta]$, we have $\gamma \in \sG_1(\Q)$ such that $\gamma^{-1} \xi=\eta$. Let $p$ be a prime number. From the assumption $\bK_{1,p}^{*}=\bK_{1,p}$ and \cite[Proposition 2.7 (ii)]{MS98}, we have the equality $$\{g\in \sG_1(\Q_p)|\,g^{-1}(\xi) \in (\cL_{1,p}^{*})_{\rm prim}\}=\sG_1^{\xi}(\Q_p)\,\bK_{1,p}.$$
 Since $u_{j,p}^{-1}\gamma^{-1}\xi=u_{j,p}^{-1} \eta \in (\cL_{1,p}^{*})_{\rm prim}$, we can find $h_p\in \sG_1^{\xi}(\Q_p)$ and $k_p \in \bK_{1,p}^{*}$ such that $\gamma_p u_{j,p} =h_pk_p$. Set $h=(h_p)_{p<\infty}\in \sG^{\xi}(\A_\fin)$ and $k:=(k_p)_{p<\infty}\in \bK_{1,\fin}$. Then we have the equality $\gamma u_{j}=hk$ in $\sG_1(\A_\fin)$. From this, we have $\j(h)=\Gamma_{Q_1}(u_j)\,\gamma^{-1}\xi=\Gamma_{Q_1}(u_j)\eta$ as desired.  

Let us prove the equality \eqref{P2} for $h \in \sG_1^\xi(\A_\fin)$ and $1\leq j\leq t$ with $\j(\bar h)\in \cL_1(u_j)^{*}$. Fix a decomposition \eqref{P1} of $h$ and set $\eta=\gamma^{-1}\xi$. Then it suffices to confirm the map $\delta \mapsto \gamma \delta\gamma^{-1}$ is a bijection from $\Gamma_{Q_1}(u_j)_{\eta}$ onto $\sG_1^{\xi}(\Q)\cap h\bK_{1,\fin}^{\xi *}h^{-1}$. Let $\delta \in \Gamma_{Q_1}(u_j)$; then we have $\delta \eta=\delta$, which is equivalently written as $g_\fin^{-1} u_j^{-1}\delta u_j g_\fin \xi=\xi$. Thus $g_\fin^{-1} u_j^{-1}\delta u_j g_\fin \in \sG_1^\xi(\A_\fin)$ on one hand. On the other hand, we have $g_{\fin}^{-1}u_j^{-1}\delta u_j g_\fin\in \bK_{1,\fin}^{*}$ due to the containment $\delta \in \Gamma_{Q_1}(u_j)$. Hence $g_\fin^{-1}u_j^{-1}\delta u_j g_\fin \in \sG_1^{\xi}(\A_\fin)\cap \bK_{1,\fin}^{*}=\bK_\fin^{\xi *}$ by \eqref{k+k+}. Therefore $\gamma \delta\gamma^{-1}=h(g_\fin^{-1}u_j^{-1}\delta u_j g_\fin)h^{-1}\in h\bK_{1,\fin}^{\xi *}h^{-1}\cap \sG_1^{\xi}(\Q)$. Hence the map $\delta \mapsto \gamma \delta\gamma^{-1}$ induces an injection from $\Gamma_{Q_1}(u_j)_{\eta}$ into $\sG_1^\xi(\Q)\cap h\bK_{1,\fin}^{\xi *}h^{-1}$. It remains to show the surjectivity of this map. For that, let $\delta_1 \in \sG_1^{\xi}(\Q)\cap h \bK_{1,\fin}^{\xi *}h^{-1}$. Then 
$$
\bK_{1,\fin}^{\xi *}\ni h^{-1}\delta_1 h=g_{\fin}^{-1}u_j^{-1}(\gamma^{-1} \delta_1\gamma )u_j g_\fin,
$$
which combined with $g_\fin\in \bK_{1,\fin}^{*}$ yields $\gamma^{-1}\delta_1\gamma \in u_jg_\fin \bK_{1,\fin}^{\xi *} g_\fin ^{-1}u_j^{-1}\subset u_j\bK_{1,\fin}^{*}u_j^{-1}$; thus $\gamma^{-1}\delta_1\delta\in \sG_1(\Q)\cap (\sG_1(\R) u_j\bK_{1,\fin}^{*}u_j^{-1})=\Gamma_{Q_1}(u_j)$. From $\delta_1 \in \sG_1^{\xi}(\Q)$, we have $\delta_1\xi=\xi$, or equivalently $\gamma^{-1}\delta_1\gamma \eta=\eta$, Hence $\delta:=\gamma^{-1}\delta_1\gamma \in \Gamma_{Q_1}(u_j)_{\eta}$ and $\delta_1=\gamma \delta \gamma^{-1}$ as desired. 
\end{proof}

Since $\xi \in \cL_1^{*}$ is supposed to be reduced, it is primitive in $\cL_1^{*}$. Since $V_1$ is isotropic by assumption, there exists a pair of isotropic vectors $\{v_0,v_0'\}$ such that $Q_1(v_0,v_0')=1$, $Q_1(v_0,\xi)=1$ and $\cL=(\Z v_0+\Z v_0')\oplus \cL_0$ with $\cL_0=\cL_1\cap \langle v_0,v_0'\rangle_\Q^{\bot}$. We introduce the following notation to write a general element of $V$:  
$$
\left[\begin{smallmatrix} x \\ y \\ z\end{smallmatrix} \right]:=xv_0+y+zv_0', \quad (x,z\in \Q,\,y\in V_0:=\langle v_0,v_0'\rangle^{\bot}_\Q).
$$
Then there exists $a\in \Z$ and $\alpha \in \cL_0^{*}$ such that 
$$
\xi=\left[\begin{smallmatrix} a \\ \alpha \\ 1\end{smallmatrix} \right]. 
$$
If we set 
$$
 [y,z]_{\xi}:=\left[\begin{smallmatrix} -z-Q_1(\alpha,y) \\ y \\ z\end{smallmatrix} \right]\quad (y\in V_0,\,z\in \Q), 
$$
then $V_1^{\xi}=\{[y,z]_{\xi}|\,y\in V_0,\,z\in \Q\}$ and 
$$Q_1([y,z]_\xi)=-2z^2-2Q_1(y,\alpha)\,z+Q_1(y).
$$
Thus we have
\begin{align}
\cL_1^{\xi}&=\{[y,z]_\xi|\,y\in \cL_0,\,z\in \Z\}, 
 \notag \\
\cL_1^{\xi *}&=\{[y,z]_{\xi}|\,Q(\cL_0,y-\alpha z)\subset \Z,\,2z+Q_1(\alpha,y)\in \Z\}.
 \label{+++}
\end{align}
Define $\tilde \sigma:V_1 \rightarrow V_1$ by demanding $\sigma(\xi)=\xi$ and 
$$
\tilde \sigma:[y,z]_\xi \mapsto [y,-z-Q_1(\alpha,y)]_{\xi}, \quad [y,z]_\xi\in V_1^\xi.$$
Then the containment $\tilde \sigma\in \sG_1^\xi(\Q)$ is confirmed by a computation. 
\begin{lem} \label{StabConnected-L}
  For any $p<\infty$, let $\tilde \sigma_p$ be the image of $\sigma$ in $\sG_1^{\xi}(\Q_p)$. Then we have $\tilde \sigma_p\in \bK_p^{\xi *}$. 
\end{lem}
\begin{proof}
From definition, $\tilde \sigma(\cL_1^\xi)\subset \cL_1^\xi$ is obvious. For any $(y,z)_\xi \in \cL_1^{\xi *}$, 
\begin{align*}
\tilde \sigma([y,z]_\xi)-[y,z]_\xi=[y,-z-Q_1(\alpha,y)]_\xi-[y,z]_{\xi}=[0,-2z-Q_1(\alpha,y)]_{\xi} \in \cL_1^{\xi}
\end{align*}
by \eqref{+++}. 
\end{proof}

\subsection{Ternary case} \label{sec:Tern}
Let 
\begin{align*}
V_1=\Bigl\{X=\left[\begin{smallmatrix} x & y \\ z & -x \end{smallmatrix}\right]\in {\Mat}_{2}(\Q)|\,\tr(X)=0\Bigr\}
, \quad Q_1(X)=-2\det X=2x^{2}+2yz. 
\end{align*}
If we identify $X=\left[\begin{smallmatrix} x & y \\ z & -x \end{smallmatrix}\right]$ with the vector $\tilde X={}^t(y,x,z)\in \Q^{3}$ then 
$$
Q_1(X)={}^t \tilde X \left[\begin{smallmatrix} 0 & 0 & 1 \\ 0 & 2 &0 \\ 1 & 0 & 0\end{smallmatrix}\right] \tilde X. 
$$
We have that $\cL_1:=V(\Z)\cong \Z^{3}$ is an integral lattice in $(V_1,Q_1)$ and 
\begin{align}
\cL_1^{*}=\{\left[\begin{smallmatrix} x & y \\ z & -x \end{smallmatrix}\right]\in {\Mat}_{2}(\Q)|y,z\in \Z,2x\in \Z\}\cong \Z\oplus 2^{-1}\Z\oplus \Z.
\label{cL1star}
\end{align}
Since $\cL_1^{*}/\cL_1\cong \Z/2\Z$, we see that $\cL_1$ is a maximal integral lattice and $\bK_{1,p}=\bK_{1,p}^{*}$ for all $p<\infty$. By letting ${\bf GL}_2$ acts on $V_1$ as 
$$
 {\bf GL}_2 \times V_1 \ni (g,X)\mapsto g Xg^{-1} \,\in V_1, 
$$
we have a $\Q$-rational isomorphism $\ss: {\bf PGL}_2 \rightarrow {\bf SO}(Q)=\sG^0$ such that 
\begin{align}
\ss(\left[\begin{smallmatrix} a & b \\ c & d\end{smallmatrix}\right])=(ad-bc)^{-1}\left[\begin{matrix} a ^2 & -2ab & -b^2 \\
 -ac & ad+bc & bd \\ 
-c^2 & 2dc & d^2 \end{matrix} \right];  
\label{ssFormula}
\end{align}
$\ss$ preserves the integral structure, i.e., ${\rm PGL}_2(\Z_p)\cong \sG_1^{0}(\Q_p)\cap \bK_{1,p}$ for all $p<\infty$. Moreover, $\sG_1=\sG_1^0 \times Z_1$, where $Z_1=\langle c^{\sG_1} \rangle$ with $c^{\sG_1}=-{\rm id}$ is the center of $\sG_1={\bf O}(Q_1)$. 

For a fundamental discriminant $D$ such that $D<0$. Set
\begin{align*}
&\xi_{D}=\left[\begin{smallmatrix} 0 & 1 \\ D/4 & 0\end{smallmatrix} \right] \quad D\equiv 0 \pmod{4}, \\
&\xi_{D}=\left[\begin{smallmatrix} 1/2 & 1 \\ (D-1)/4 & -1/2\end{smallmatrix} \right] \quad D\equiv 1 \pmod{4}.
\end{align*}

\begin{lem} \label{L-GxiDclassBiQ}
We have that $Q_1(\xi_D)=D/2$ and $\xi_D\in \cL_1^*$ is a reduced vector. We have
\begin{align*}
\sG_1^{\xi_D}(\Q)\bsl \sG_1^{\xi_D}(\A_\fin)/\bK_{1,\fin}^{\xi_D *}\cong {\bf SL}_2(\Z)\bsl \cQ^{+}_{\rm{prim}}(D),
\end{align*}
where 
$$
 \cQ^{+}_{\rm{prim}}(D)=\Bigl\{\left[\begin{smallmatrix} b & a/2 \\ a/2 & c \end{smallmatrix}\right] |b,c,a\in \Z,\,b>0,\,(a,b,c)=1,\,{a^2}-4bc=D\Bigr\}
$$
on which ${\bf SL}_2(\Z)$ acts by ${\bf SL}_2(\Z) \times \cQ^{+}_{\rm{prim}}(D)\ni (\gamma,T)\mapsto \gamma T{}^t \gamma\in \cQ^{+}_{\rm{prim}}(D)$. 
\end{lem}
\begin{proof} $Q_1(\xi_D)=D/2$ is confirmed by a computation. From ${\bf GL}_2(\A)={\bf GL}_2(\Q){\bf GL}_2(\R){\bf GL}_2(\hat \Z)$, we have 
$$
\sG_1^0(\A)=\sG_1^{0}(\Q)\sG_1^0(\R) (\sG_1^0(\A_\fin)\cap \bK_{1,\fin}).  
$$
Since $Z_1(\A_\fin)\subset \bK_{1,\fin}$, this gives us 
$$
\sG_1(\A)=\sG_1(\Q)\sG_1(\R)\bK_{1,\fin}.
$$
Thus from Proposition~\ref{Prop1}, 
\begin{align}
\sG_1^{\xi_D}(\Q)\bsl \sG_1^{\xi_D}(\A_\fin)/\bK_{1,\fin}^{\xi_D*}\cong \Gamma_{Q_1} \bsl \cL^{*}_{1,{\rm prim},[D/2]},
 \label{L-GxiDclassBiQ-f1}
\end{align}
where
$$
\Gamma_{Q_1}=\{g\in \sG_1(\Q)|\,u\cL_1=\cL_1\}.
$$
Let
$$
\cQ=\{\left[\begin{smallmatrix} b & a/2 \\ a/2 & c \end{smallmatrix}\right]\in {\Mat}_{2}(\Q)|b,c,a\in \Z\}
$$
identified with the space of integral binary quadratic forms $[b,a,c]=bx^2+axy+cy^2$ and $\cQ_{\rm prim}$ the space of primitive integral binary quadratic forms $[b,a,c]$ (${\rm gcd}(a,b,c)=1$). 
The map
$$
i:X\rightarrow X w, \quad w=\left[\begin{smallmatrix} 0 & 1 \\ -1 & 0 \end{smallmatrix}\right]
$$
yields $i:\cL_1^{*} \underset{\cong}{\rightarrow} \cQ$ such that 
$$
 i(gXg^{-1})=(\det g)^{-1}\,g\,i(X)\,{}^t g, \quad g\in {\bf GL}_2(\Z).
$$
Let $Q_1'$ be the quadratic form on $\cQ$, the transform of $Q_1$ by $i$; then $Q_1'(\left[\begin{smallmatrix} b & a/2 \\ a/2 & c \end{smallmatrix}\right])=-2\det(\left[\begin{smallmatrix} b & a/2 \\ a/2 & c \end{smallmatrix}\right]w)=-2(bc-\frac{a^2}{4})$. We have $i(\cL_{1,{\rm prim},[D/2]}^*)=\cQ_{\rm prim}(D)$, where $\cQ_{\rm prim}(D):=\{T\in \cQ_{\rm prim}|\,Q'_1(T)=D/2\}$. By \eqref{L-GxiDclassBiQ-f1}, it suffices to show that $i$ induces a bijection 
$$
\Gamma_{Q_1}\bsl \cL_{1,{\rm prim},[D/2]}^{*} \cong {\bf SL}_2(\Z)\bsl \cQ_{\rm prim}^{+}(D).
$$ We have
$$
 \Gamma_{Q_1}\underset{\ss}{\cong} {\bf GL}_2(\Z)/\{\pm 1_2\}\ltimes \{1,\tilde c \}
$$
by defining $\ss(\tilde c)=c^{\sG}$. By the map induced from $i$, the orbit space $\Gamma_{Q_1}\bsl \cL_{{\rm prim},[D/2]}^{*}$ is identified with the ${\bf GL}_2(\Z)\ltimes \{1,\tilde c\}$-equivalence classes in $\cQ_{\rm prim}(D)$ where $\gamma \in {\bf GL}_2(\Z)$ acts on $\cQ$ as $X\mapsto \det(\gamma)\,\gamma X{}^t \gamma$ and $\tilde c$ acts on $\cQ$ as $X\mapsto -X$. Since
$$
({\bf GL}_2(\Z) \ltimes\{1,\tilde c\}) \bsl \cQ_{\rm{prim}}(D)\cong {\bf SL}_2(\Z)\bsl \cQ_{\rm{prim}}^{+}(D), 
$$
we are done. 
\end{proof}
Let $E=\Q(\sqrt{D})$ be the quadratic extension of discriminant $D<0$. Set
$$
\omega=
\begin{cases}
\frac{\sqrt{D}}{2} \quad &(D\equiv 0 \pmod{4}), \\
\frac{\sqrt{D}-1}{2} \quad &(D\equiv 1 \pmod{4}).
\end{cases}
$$
Then $\{1,\omega\}$ is a $\Z$-basis of the integer ring $\cO_E$ of $E$, i.e., $\cO_E=\Z\oplus \Z\omega$. Set $w=\left[\begin{smallmatrix} 0 & 1 \\ -1 & 0 \end{smallmatrix}\right]$ and $T_{D}=\xi_{D}w^{-1}$. For $\alpha \in E$, its conjugate is denoted by $\bar \alpha$. Then a computation reveals that the relation
$$
(X+ \omega Y)(X+\bar \omega\,Y)=[X,Y] T_D \left[\begin{smallmatrix} X \\ Y \end{smallmatrix}\right] 
$$
holds in the polynomial ring $\C[X,Y]$, where $\{X,Y\}$ is a set of indeterminates. We have an embedding $\iota: E^\times \rightarrow {\bf GL}_2$ such that 
\begin{align}
[\tau,\tau\omega]=[1,\omega]\,{}^t(\iota(\tau)), \quad \tau \in E^\times,
 \label{embedE}
\end{align}
whose image coincides with 
$$
{\bf GO}(T_D)^{0}=\{h\in {\bf GL}_2|\,hT_D{}^t h=(\det h)\,T_{D}\}=\{h\in {\bf GL}_2|\,\ss(h)\xi_{D}=\xi_{D}\}. 
$$
Indeed, set $h=\iota(\tau)$ and put $X'=h_{11}X+h_{21}Y$, $Y'=h_{12}X+h_{22}Y$, i.e., $[X',Y']=[X,Y]\,h$. Then, from \eqref{embedE}, 
\begin{align*}
\nr(\tau) [X,Y]T_{D}\left[\begin{smallmatrix} X \\ Y \end{smallmatrix}\right]
&=(\tau X+\tau \omega Y)(\tau X+\bar\tau \bar \omega Y) \\
&=\{(h_{11}+h_{12}\omega)X+(h_{21}+h_{22}\omega)Y\}\{(h_{11}+h_{12}\bar \omega)X+(h_{21}+h_{22}\bar \omega)Y\}
\\
&=\{(h_{11}X+h_{21}Y)+\omega(h_{12}X+h_{22}Y)\}\{(h_{11}X+h_{21}Y)+\bar\omega(h_{12}X+h_{22}Y)\}
\\
&=(X'+\omega Y')(X'+\bar \omega Y')
= [X',Y']T_{D}\left[\begin{smallmatrix} X '\\ Y' \end{smallmatrix}\right]
\\
&= [X,Y]h T_{D}\,{}^th \left[\begin{smallmatrix} X \\ Y \end{smallmatrix}\right].
\end{align*}
Therefore, 
$$
\nr(\tau)T_{D}=h T_{D}\,{}^th, \quad \det h=\nr(\tau). 
$$
The composite of the isomorphisms $\iota:E^\times \rightarrow {\bf GO}(T_D)^{0}$ and $\ss:{\bf PGL}_2\rightarrow {\bf SO}(Q_1)=\sG_1^{0}$ induces an isomorphism 
$$
\ss\circ \iota: E^{\times}/\Q^\times \underset{\iota}{\cong }{\bf PGO}(T_{D})^{0} \underset{\ss}{\cong} {\bf SO}(Q_1)_{\xi_D}=\sG_1^{0}\cap \sG_1^{\xi_D}=(\sG_1^{0})^{\xi_D}. 
$$
\begin{lem} \label{L-CLd-stab}
 The map $\ss\circ \iota$ induces a bijection
$$
\A_{E,\fin}^{\times}/E^\times {\widehat \cO_E}^\times \cong \sG_1^{\xi_D}(\Q)\bsl \sG_1^{\xi_D}(\A_\fin)/\bK_{1,\fin}^{\xi_D *}.
$$
\end{lem}
\begin{proof}
Let $p$ be a prime. From \eqref{embedE}, we have $\cO_{E,p}^{\times}=\iota^{-1}({\bf GL}_2(\Z_p))$. Since $\ss({\bf GL}_2(\Z_p))=\sG_1^{0}(\Q_p)\cap \bK_{1,p}$, we have
$$
\cO_{E,p}^{\times}/\Z_p^\times \cong (\sG_1^{0})^{\xi_D}(\Q_p)\cap \bK_{1,p}^{}=\bK_{1,p}^{\xi_D *}\cap \sG_1^0(\Q_p). 
$$
From Lemma~\ref{StabConnected-L}, there exists a $\tilde \sigma\in \sG^{\xi_D}_1(\Q)-(\sG_1^{\xi_D})^{0}(\Q)$ such that $\tilde \sigma_p \in \bK_1^{\xi_D *}$. 
\begin{align}
\bK_{p,1}^{\xi_D *}=\bK_{1,p}^{\xi_D *}\cap \sG_1^0(\Q_p)\,\{1,\tilde \sigma_p\}.
\label{L-CLd-stab-f1}
\end{align}
Since $\Q$ is of class number $1$, $\A^\times=\Q^\times \R_{>0}\prod_{p<\infty}\Z_p^\times$. We have
\begin{align*}
\A_{E}^\times/E^\times \,\C^\times\, \widehat{\cO_E}^{\times} 
&\cong \A_{E}^{\times}/E^\times\,\A^\times\, \C^\times\, \widehat{\cO_{E}}^{\times} \\
&\cong \A_{E,\fin}^{\times}/E^\times\,\A_{\fin}^\times\, \widehat{\cO_{E}}^{\times} 
\\
&\underset{\ss\circ\iota}{\cong} (\sG_1^{0})^{\xi_D}(\Q)\bsl (\sG_1^0)^{\xi_D}(\A_\fin)/\prod_{p<\infty}(\sG_1^0(\Q_p)\cap \bK_{1,p}^{\xi_D *}) 
\\
&\cong \sG_1^{\xi_D}(\Q)\bsl \sG_1^{\xi_D}(\A_\fin)/\bK_{1,\fin}^{\xi_D *}
\end{align*}
by using \eqref{L-CLd-stab-f1} to have the last isomorphism. 
\end{proof}

\begin{lem} \label{GxiD-class}
 Let $h_D$ be the class number of $E=\Q(\sqrt{D})$ and $J=\{u_1,\dots,u_{h_D}\}$ a complete set of representatives in $\A_{E,\fin}^\times $ modulo $E^{\times}\widehat{\cO_{E}}^\times$. Let $\sigma': j\mapsto \hat{j}$ be the involution of $J$ defined as $\bar {u_j}\equiv u_{\hat{j}} \pmod{E^\times \widehat {\cO_{E}}^\times}$. Let $\cJ$ be a complete set of representatives of $J/\{{\rm id},\sigma'\}$. Set $\tilde u_{j}=\ss\circ \iota(u_j)\in \sG_1^{\xi_D}(\A_\fin)$. Then $\{\tilde u_j\}_{j\in \cJ}$ yields a complete set of representatives of $\sG_1^{\xi_D}(\Q)\bsl \sG_1^{\xi_D}(\A_\fin)/\bK_{1,\fin}^{\xi_D *}$. Moreover, for $j\in \cJ$, 
$$
e_j:=\#(\sG_1^{\xi_D}(\Q)\cap \tilde u_j\bK_{1,\fin}^{\xi_D *}\tilde u_j^{-1})
=\{1+\delta(j=\hat j)\}\frac{w_{D}}{2},
$$
where $w_{D}=\# \cO_{E}^\times$ and the total volume of $\sG^{\xi_D}_1(\Q)\bsl \sG_1^{\xi_D}(\A_\fin)$ is 
$$
\mu_D:=\sum_{j\in \cJ}e_j^{-1}=\frac{h_D}{w_D}.
$$
\end{lem}
\begin{proof} Recall that $\sG_1^{\xi_D}(\Q)={\Im }{\ss}\,\{1,\tilde \sigma\}$. Let $\sigma$ denote the non-trivial automorphism of $E/\Q$. The embedding $\ss\circ \iota$ from $E^\times/\Q^\times$ to $\sG_1^{\xi_D}$ is extended to $E^\times/\Q^\times\,\{1,\sigma\}$ by setting $(\ss\circ \iota)(\sigma)=\tilde \sigma$. Let $h=(\ss\circ\iota(t\tau))$ with $t\in {E}^{\times}/\Q^\times $ and $\tilde{\tau}\in \{1, \sigma\}$. Then $h \in \tilde u_j \bK_{\fin}^{\xi_D *}\tilde u_j^{-1}$ if and only if
 $$
u_j(t\tau)u_j^{-1} \in \widehat{\cO_E}^{\times}\,\Sigma,
$$
where $\Sigma=\prod_{p<\infty}\{1,\sigma_p\}$ with $\sigma_p$ a copy of $\sigma$ identified with the unique non-trivial automorphism of $E_p=E\otimes_{\Q}\Q_p$ over $\Q_p$. Since $\sigma u_j \sigma=\bar u_j$, this is equivalent to 
\begin{itemize}
\item[(i)] $\tau=1, \quad t\in \widehat {\cO_{E}}^{\times} $, or 
\item[(ii)] $\tau=\sigma, \quad t u_j\bar u_j^{-1} \in \widehat {\cO_E}^{\times}$.
\end{itemize}
When we have the case (i), then $t\in \cO_E^\times/\{\pm 1\}$. The case (ii) happens if and only if $u_j {\bar u_j}^{-1}\in E^\times {\widehat \cO_E}^{\times}$, or equivalently $j= \hat{j}$; then $t \in \cO_{E}^\times/\{\pm 1\}$. Hence $e_j=\{1+\delta(j=\hat j)\}w_D/2$. We have 
$$
\sum_{j \in \cJ}(1/e_j)=2w_D^{-1}\biggl(\#\{j\in \cJ|\,j\not=\hat j\}+\tfrac{1}{2}\#\{j\in \cJ|\,j=\hat j\} \biggr)
=\frac{h_D}{w_D}. 
$$
\end{proof}

Let $\cV(\xi_D)$ be the space of all those smooth functions on $\sG_1^{\xi_D}(\A)$ such that $f(\delta h u_\infty)=f(h)$ for all $\delta\in \sG_1^{\xi_D}(\Q)$, $h \in \sG_1^{\xi_D}(\A)$ and $u_\infty \in \sG_1^{\xi_D}(\R)$. Let $\cV(\xi_D;\bK_{1,\fin}^{\xi_D*})$ be the space of $\bK_{1,\fin}^{\xi_D*}$-fixed vectors in $\cV(\xi_D)$. Since $2\xi_D\in \cL_{1}$, an involutive operator $\tau_\fin^{\xi_D}$ on $\cV(\xi_D,\bK_{1,\fin}^{\xi_D*})$ is defined as $[\tau_{\fin}^{\xi_D}f](h)=f(h h_\fin^{\xi_D})$ with $h_{\fin}^{\xi_D}\in \sG_1^{\xi_D}(\A_\fin)$ any element such that ${\mathsf r}^{\xi_D}\in h_{\fin}^{\xi_D} \bK_{1,\fin}^{*}$ where ${\mathsf r}^{\xi_D}$ is the reflection of $V_1$ with respect to the vector $\xi_D$ (see \cite[\S2.9]{Tsud2019}). 

\begin{lem} \label{L:tauxiD}
$\tau_{\fin}^{\xi_D}$ is the identity map. 
\end{lem}
\begin{proof}
Let ${\mathsf c}^{\xi_D}$ (resp. ${\mathsf c}_1$) be the non-trivial elements of the center of $\sG_1^{\xi_D}(\Q)$ (resp. $\sG_1(\Q)$). Then ${\mathsf r}^{\xi_D}={\mathsf c}^{\xi_D}{\mathsf c}_1$. We claim that ${\mathsf c}_1$ viewed as an element of $\sG_1(\A_\fin)$ belongs to $\bK_{1,\fin}^{*}$. Indeed, since $2\cL_1^{*}\subset \cL_1$ by \eqref{cL1star}, we have ${\mathsf c}_1(X)-X=-X-X=-2X\in \cL_{1}$ for all $X\in \cL_{1}^{*}$. Therefore for $f\in \cV(\xi_D,\bK_{1,\fin}^{\xi_D*})$, we have $[\tau_{\fin}^{\xi_D}f](h)=f(h{\mathsf c}^{\xi_D})=f({\mathsf c}^{\xi_D}h)$, which equals to $f(h)$ due to ${\mathsf c}^{\xi_D}\in \sG_1^{\xi_D}(\Q)$ and to the automorphy of $f$. 
\end{proof}

Set ${\rm E}_p(\xi_D):=\bK_{1,p}^{\xi_D}/\bK_{1,p}^{\xi_D*}$ for a prime number $p$. 
\begin{lem} \label{Lem:Ep}
If $p$ is inert or splits in $\Q(\sqrt{D})/\Q$, then ${\rm E}_p(\xi_D)=\{1\}$. If $p$ ramifies in $\Q(\sqrt{D})/\Q$, then ${\rm E}_p(\xi_D)\cong \Z/2\Z$. 
\end{lem}
\begin{proof}
If $E_p=\Q_p(\sqrt{D})$ is a ramified field extension of $\Q_p$, then $\bK_{1,p}^{\xi_D}=\sG_1^{\xi_D}(\Q_p)\cong (E_p^\times/\Q_p^\times) \rtimes {\rm Gal}(E_p/\Q_p)$ and $\bK_{1,p}^{\xi_D*}\cong (\cO_{E,p}^{\times}/\Z_p^\times)\rtimes {\rm Gal}(E_p/\Q_p)$ from the proof of Lemma~\ref{L-CLd-stab}. Let $\varpi_p$ be a prime element of $E_p$; then ${\rm E}_p(\xi_D)\cong E_p^\times/\Q_p^\times \cO_{E,p}^\times$ is represented by the class of $1$ and $\varpi_p$. Thus ${\rm E}_p(\xi_D)\cong \Z/2\Z$. 
\end{proof}

For a unitary character $\chi$ of the finite group $\A_{E,\fin}^{\times}/E^\times \widehat {\cO_E}^{\times}\cong {\rm Cl}_D$, define a function $f_\chi$ on $\sG_1^{\xi_D}(\A_\fin)\cong (\A_{E,\fin}^{\times}/\A_{\fin}^{\times})\rtimes \Sigma$ by setting 
\begin{align}
f_{\chi}(\ss\circ\iota(t\tau))=\tfrac{1}{2}\{\chi(t)+\chi(\bar t)\}, \quad t\in \A_{E,\fin}^{\times}, \tau\in \Sigma:=\prod_{p<\infty}\{1,\sigma_p\}. 
 \label{def:fchi}
\end{align}

\begin{lem} \label{L:IrredDecomp-cV}
The function $f_\chi$ belongs to the space $\cV({\xi_D};\bK_{1,\fin}^{\xi_D *})$ and is a joint eigenfunction of the Hecke algebra $\cH^{+}(\sG_1^{\xi_D}(\A_\fin)\sslash\bK_{1,\fin}^{\xi_D *})$. Let $\widehat{{\rm Cl}_D}/{\rm Gal}(E/\Q)$ be the Galois equivalence classes in $\widehat{{\rm Cl}_D}$. The set of functions $f_{\chi}\,(\chi \in \widehat {\rm {Cl}_D}/{\rm Gal}(E/\Q))$ forms an orthogonal basis of $\cV({\xi_D};\bK_{1,\fin}^{\xi_D *})$ such that
\begin{align*}
\|f_\chi\|_{\sG_1^{\xi_D}}^{2}=\frac{h_D}{2w_D} \times \{1+\delta(\chi^{2}={\bf 1})\}.
\end{align*}
Let $\cU_\chi$ be the $\sG_1^{\xi_D}(\A_\fin)$-submodule generated by $f_\chi$. Then $\cU_{\chi}$ is irreducible and the space of $\bK_{1,\fin}^{\xi_D*}$-fixed vectors in $\cU_\chi$ coincides with $\C f_\chi$. The map $\chi\mapsto \cU_\chi$ yields a bijection between $\widehat{{\rm Cl}_D}/{\rm Gal}(E/\Q)$ and the set of all the irreducible $\sG_1^{\xi_D}(\A_\fin)$-submodules in $\cV(\xi_D)$ with $\bK_{1,\fin}^{\xi_D*}$-fixed vectors. The $L$-function $L_\fin(s,\cU_\chi)$ of $\cU_\chi$ coincides with Hecke's $L$-function $L_\fin(s,{\mathcal{AI}}(\chi))$ of ${\mathcal{AI}}(\chi)$. If $\chi={\bf 1}$ is the trivial character, then $L_\fin(s,\cU_{{\bf 1}})=\zeta(s)\,L_\fin(s,\eta_D)$. 
\end{lem}
\begin{proof}
The containment $f_\chi\in \cV(\xi_D,\bK_{1,\fin}^{\xi_D*})$ is easy to be checked by \eqref{def:fchi}. Let $C_{\rm c}(E_p^\times/\cO_{E,p}^{\times})^{+}$ be the convolution algebra of all $\C$-valued $\cO_{E,p}^\times$-invariant compactly supported functions $\phi_0$ on $E_p^\times$ such that $\phi_0(\bar t)=\phi_0(t)$ ($t\in E_{p}^\times)$. For $\phi_0\in C_{\rm c}^{\infty}(E_p^\times/\cO_{E,p}^\times)^{+}$, define $\phi\in \cH(\sG_1(\Q_p)\sslash \bK_{1,\fin}^{\xi_D*})$ by $\phi(t\tau)=\phi_0(t)$ $(t\in E_{p}^{\times}/\Q_p,\,\tau\in {\rm Gal}(E_p/\Q_p))$. Then $\phi_0 \mapsto \phi$ yields a $\C$-algebra isomorphism from $C_{\rm c}(E_p^\times/\Q_p)^{+
}$ to $\cH_p:=\cH(\sG_1^{\xi_D}(\Q_p)\sslash \bK_{1,p}^{\xi_D*})$. In particular, $\cH_p$ is commutative so that its center $\cH_p^{+}$ coincides with $\cH_p$ itself. By this description of $\cH_p^{+}$, it is easy to check that $f_\chi$ is a joint-eigenfucntion of $\cH_p^{+}$ for all $p$. From \cite[Proposition 13.1]{Tsud2019}, the $\bK_{1,\fin}^{\xi_D*}$-fixed Hecke eigenvector $f_\chi$ generates an irreducible $\sG_1^{\xi_D}(\A_\fin)$-submodule of $\cV(\xi_D)$. The $L$-function $L(s,\cU_\chi)$ is defined to be $L(s,f_\chi)$ whose definition is given in \cite[\S 1.4]{MS98}. Let $S_{E}$, $I_E$ and $R_E$ the set of $p\in \fin$ which splits, remains inert or ramifies in $E/\Q$, respectively. Since ${\rm E}_p(\xi_D)=\bK_{1,p}^{\xi_D}/\bK_{1,p}^{\xi_D*}$ is isomorphic to $\{1\}$ or $\Z/2\Z$ according to $p\in S_E\cup I_E$ or $p\in R_E$ respectively (Lemma~\ref{Lem:Ep}), the set of Satake parameters $\{(z_p,\rho_p)\}_{p\in S_E}\cup \{\rho_p\}_{p\in R_{E}\cup I_E}$ of $f_\chi$ (in the extended sense of \cite{MS98}) is described as follows. If $p\in S_E$, then $E_p^\times \cong \Q_p^\times \oplus \Q_p^\times $ and $\chi_p=\chi_p'\boxtimes \chi_p''$ with unramified characters $\chi_p'$ and $\chi_p''$ such that $\chi_p'\chi_p''=1$, and ${\rm E}_p(\xi_D)=\{1\}$. We have\begin{align*}
z_p=(\chi_p'(p),\chi_p''(p)), \quad \rho_p=1 
\end{align*}
and $L_p(s,f_\chi)=(1-\chi_p'(p)p^{-s})^{-1}(1-\chi''_p(p)p^{-s})^{-1}$. If $p \in I_E$, then $\sG_1^{\xi_D}$ is anisotropic and unramified over $\Q_p$. Hence the Satake parameter of $f_\chi$ at $p$ is a unique character of ${\rm E}_p(\xi_D)=\{1\}$. This falls in the case $(n_0,\partial)=(2,0)$ of \cite[(1.18)]{MS98}; thus $L_p(s,f_\chi)=(1-p^{-2s})^{-1}$. If $p \in R_E$, then $\sG_1^{\xi_D}$ is anisotropic over $\Q_p$ and the Satake parameter of $f_\chi$ is a character $\rho_p$ of ${\rm E}_p(\xi_D)\cong \Z/2\Z$; $\rho_p=1$ if $\chi_p(\varpi_p)=1$ and $\rho_p$ is the nontrivial character of $\Z/2\Z$ if $\chi_p(\varpi_p)=-1$\, where $\varpi_p$ is a prime element of $E_p$. This falls in the case $(n_0,\partial)=(2,1)$ in \cite[(1.18)]{MS98}; thus $L_p(s,f_\chi)=(1-\chi_p(\varpi_p)p^{-s})^{-1}$. To sum up all the cases, we have $L_\fin(s,f_\chi)=L_\fin(s,\chi)$. 

Recall $\sG_1^{\xi_D}(\A_\fin)\cong (\A_{E,\fin}^{\times}/\A_\fin^\times)\rtimes \Sigma$, where $\Sigma=\prod_{p\in \fin}\{1,\sigma_p\}$ acts on $\A_{E}$ by coordinate-wise Galois conjugation. We endow the compact group $\Sigma$ with the probability Haar measure; then there exists a unique Haar measure on $\A_{E,\fin}^{\times}/\A_\fin^\times$ which matches the Haar measures on $\sG_1^{\xi_D}(\A_\fin)$ and on $\Sigma$. Since a natural map from $(\sG_1^{\xi_D})^{0}(\Q)\bsl \sG_1^{\xi_D}(\A_\fin)$ to $\sG_1^{\xi_D}(\Q)\bsl \sG_1^{\xi_D}(\A_\fin)$ is two-to-one and since $(\sG_1^{\xi_D})^0(\Q)\bsl \sG_1^{\xi_D}(\A_\fin)\cong (\A_{E,\fin}^\times /E^{\times}\A_\fin^\times) \rtimes \Sigma$, the inner product of $f_{\chi}$ and $f_{\eta}$ is computed as
{\allowdisplaybreaks\begin{align*}
\langle f_{\chi},f_{\eta}\rangle_{\sG_1^{\xi_D}}
&=\int_{\sG_1^{\xi_D}(\Q)\bsl \sG_1^{\xi_D}(\A)} f_{\chi}(h)\bar f_{\eta}(h)\,\d h
\\
&=\tfrac{1}{2}\int_{(\sG_1^{\xi_D})^{0}(\Q)\bsl \sG_1^{\xi_D}(\A_\fin)}f_{\chi}(h)\bar f_{\eta}(h)\,\d h
\\
&=\tfrac{1}{2} \int_{\A_{E,\fin}^\times/E^\times \A^\times_\fin } \int_{\Sigma} f_{\chi}(\ss\circ \iota(t\tau))\bar f_{\eta}(\ss \circ \iota(t\tau))\,\d t\,\d \tau
\\
&=\tfrac{1}{2}\int_{\A_{E,\fin}^\times/E^\times\A_\fin^{\times}} \tfrac{1}{2}(\chi(t)+\chi(\bar t))\times \tfrac{1}{2}(\eta(t)
+\eta(\bar t)))\,\d t\\
&=\tfrac{1}{4}{\vol(\A_{E,\fin}^\times/E^\times\A_\fin^\times)}(\delta(\chi=\eta)+\delta(\chi=\eta^{\sigma})).
\end{align*}}From our choice of the Haar measures, $\vol(\A_{E,\fin}^\times/E^\times \A_\fin^\times)=2\vol(\sG_1^{\xi_D}(\Q)\bsl \sG_1^{\xi_D}(\A_\fin))$; thus $\vol(\A_{E,\fin}^\times/E^\times \A_\fin^\times)=2h_D/w_D$ form Lemma~\ref{GxiD-class}. Thus $f_{\chi}\,(\chi\in \widehat{{\rm Cl}_D}/{\rm Gal}(E/\Q))$ is orthogonal. Note that $\chi=\chi^{\sigma}$ if and only if $\chi^{2}={\bf 1}$ as observed in \S~\ref{sec:Intro}. From Lemma~\ref{GxiD-class}, $\#(\widehat{{\rm Cl}_D}/{\rm Gal}(E/\Q))=\#(\sG_1^{\xi_D}(\Q)\bsl \sG_1^{\xi_D}(\A_\fin)/\bK_{1,\fin}^{\xi_D*})=\dim \cV(\xi_D,\bK_{1,\fin}^{\xi_D*})$. Hence $f_{\chi}$ forms an orthogonal basis of $\cV(\xi_D,\bK_{1,\fin}^{\xi_D*})$. Then the statements on the representations $\cU_\chi$ follow from \cite[Proposition 13.1]{Tsud2019}. 
\end{proof}

\section{Asymptotic formula for orthogonal group of degree $5$} \label{sec:Othogonal} 
First we recall the notation and main result from \cite{Tsud2019} in a special setting. 
Let $\Q^5$ be the space of column vectors of degree $5$ viewed as a quadratic space with the quadratic form ${}^{t}XQY$, where  
\begin{align}
Q=
\left[
\begin{smallmatrix} 
{} & {} &{} & {}&  1 \\
{} & {} &{} & {1}&  {} \\
{} & {} &{2} & {}& {}  \\
{} & {1} &{} & {}&  {} \\
{1} & {} &{} & {}&  {}
\end{smallmatrix}\right]. 
 \label{matrixQ}
\end{align}
The standard basis of $\Q^5$ is labeled as $\e_1,\e_0,v,\e_0',\e_1'$ in this section. Set $\cL=\Z^5$. Then the dual lattice $\cL^*$ of $\cL$ is given as 
$$
\cL^*=\Z\e_1\oplus \Z\e_0\oplus (2^{-1}\Z)v\oplus \Z\e_0'\oplus \Z\e_1'.
$$
Let $\sG={\bf O}(Q)$ and $\bK_\fin=\prod_{p<\infty}\bK_p$ with $\bK_p=\sG(\Q_p)\cap {\bf GL}_{5}(\Z_p)$. Since the group $\cL^{*}/\cL\cong \Z/2\Z$ admits no non-trivial group automorphism, we have that $\bK_\fin^*:={\rm Ker}(\bK_\fin\rightarrow {\rm Aut}(\cL^*/\cL))$ coincided with $\bK_\fin$. 

Set
$$
{}^t[x_1,X,y_1]=\left[\begin{smallmatrix} x_1 \\ X \\ y_1 \end{smallmatrix}\right] :=\left[\begin{smallmatrix} x_1 \\ b \\ a \\ c \\ y_1 \end{smallmatrix} \right], \quad x_1,y_1\in \Q,\,X=\left[\begin{smallmatrix} a & b \\ c & -a \end{smallmatrix} \right]\in V_1(\Q). 
$$
Then the quadratic space $(V_1,Q_1)$ considered in \S~\ref{sec:Tern} is isometrically embedded to $(\Q^5,Q)$ by the map sending $X\in V_1$ to the vector ${}^t[0, X,0]\in \Q^5$. Here, we remind the readers that an element $X=\left[\begin{smallmatrix} x_1 & x_2 \\ x_3 & -x_1\end{smallmatrix}\right]$ of $V_1$ is identified with a column vector ${}^t[x_2,x_1,x_3]$ and also with a symmetric matrix $Xw^{-1}=\left[\begin{smallmatrix} x_2 & -x_1 \\ -x_1 & -x_3 \end{smallmatrix}\right]\in \cQ$ from time to time. Set $\fz_0=\left[\begin{smallmatrix} 0 & \sqrt{-1} \\ -\sqrt{-1} & 0  \end{smallmatrix}\right] \in V_1(\C)$. Let $\cD$ be the connected component of $\tilde \cD:=\{\fz \in V(\C)|\,Q_1[\Im(\fz)]<0\}$ containing the point $\fz_0$, or explicitly
$$
\cD=\left\{\fz=\left[\begin{smallmatrix}z_2&  z_1 \\ z_3 & -z_2 \end{smallmatrix} \right]\in \C^3\bigm|\,(\Im z_1)(\Im z_3)+(\Im z_2)^{2}<0, \,\Im z_1>0\,\right\}.
$$
The group $\sG(\R)$ acts on $\tilde \cD$ as $\sG(\R)\times \tilde \cD\ni (g,Z)\mapsto g\langle \fz \rangle \in \tilde \cD$, where 
\begin{align}
g\left[\begin{smallmatrix} -Q_1[\fz]/2 \\ \fz \\ 1 \end{smallmatrix} \right]
=J(g,\fz)\,\left[\begin{smallmatrix} 
-Q_1[g\langle \fz \rangle]/2 \\ g\langle \fz \rangle \\ 1 \end{smallmatrix}\right]\label{ActionGD}
\end{align}
with $J(g,\fz) \in \C^\times$ the factor of automorphy. Let $\sG(\R)^{+}=\{g\in \cG(\R)|\,g\langle \cD\rangle=\cD\}$. Then $\sG(\R)^{+}$ is a normal subgroup of $\sG(\R)$ of index $2$ such that $\sG(\R)^{0} \subset \sG(\R)^{+}$. Set $\sG(\Q)^{+}=\sG(\Q)\cap \sG(\R)^{+}$. 

For an even positive integer $l$, Let $S_{l}(\bK_{\fin})$ be the space of all those holomorphic bounded functions ${\rm F}:\cD\times \sG(\A_\fin) \rightarrow \C$ such that 
\begin{align}
{\rm F}(\gamma\langle \fz\rangle,\gamma g_\fin k)=J(\gamma,\fz)^{l}{\rm F}(\fz,g_\fin), \quad \gamma \in \sG(\Q)^{+},\,(\fz,g_\fin)\in \cD\times \sG(\A_\fin),\,k\in \bK_\fin.
 \label{Gautomorphy}
\end{align}
For our particular $\sG$, we have $\sG(\A_\fin)=\sG(\Q)^{+}\bK_{\fin}$. Hence for any $g_\fin \in \sG(\A_\fin)$, we have ${\rm F}(\fz,g_\fin)={\rm F}(\gamma\langle \fz\rangle, 1)$ from \eqref{Gautomorphy} by writing $g_\fin=\gamma k$ with $\gamma\in \sG(\Q)^+$ and $k\in \bK_\fin$. Thus we can identify $S_l(\bK_\fin)$ with the space of bounded holomorphic functions $F:\cD\rightarrow \C$ such that $F(\gamma\langle \fz \rangle)=J(\gamma,\fz)^{l}F(\fz)$ for all $\gamma \in \Gamma^{+}(Q)$, where we set $\Gamma^{+}(Q)=\sG(\Z)\cap \sG(\Q)^+$.

Let $\cL_1^{*}\cong\Z\oplus 2^{-1}\Z\oplus \Z$ be the dual lattice of $\cL_1=V_1(\Z)$ as in \S~\ref{sec:Tern}. Let $a_{{\rm F}}(g_\fin; n)\,(g_\fin \in \sG(\A_\fin),\,\eta \in \cL_1^{*},\,Q_1[\eta]<0)$ be the set of Fourier coefficients of ${\rm F}$, which fits in the Fourier series expansion of ${\rm F}$:$$
{\rm F}(\fz,g_\fin)=\sum_{\substack{\eta \in \cL_1^{*} \\
Q_1[\eta]<0}} 
a_{\rm F}(g_\fin; \eta)\,\exp(2\pi \sqrt{-1}(z_1\eta_3+2z_2\eta_2+z_3\eta_1)) 
, \quad \fz=\left[\begin{smallmatrix}z_2 & z_1 \\z_3 & -z_2\end{smallmatrix}\right]\in \cD, \,g_\fin \in \sG(\A_\fin).
$$
The Hecke algebra $\cH(\sG(\A_\fin)\sslash\bK_\fin)$ acts on a modular form ${\rm F}(\fz,g_\fin)$ through the convolution product in the second variable $g_\fin$. Fix an orthogonal basis $\cF_l$ of $S_l(\bK_\fin)$ consisting of joint eigenfunctions of Hecke operators from $\cH(\sG(\A_\fin)\sslash \bK_\fin)$, where the inner product of $S_l(\bK_\fin)$ is defined as 
$$
\langle {\rm F},{\rm F}_1\rangle=\int_{\sG(\Q)^{+}\bsl (\cD \times \sG(\A_\fin))}{\rm F}(\fz,g_\fin)\,\overline{F_1(\fz, g_\fin)}\,\d \mu_{\cD}(\fz)\,\d g_\fin$$
with $\d\mu_{\cD}(\fz)$ a $\sG(\R)^0$ invariant measure on $\cD$ given as 
\begin{align}
 \d\mu_{\cD}(\fz)=(Q_1(\Im(\fz))^{-3}\,\prod_{j=1}^{3}2^{-1}|\d z_j\wedge \d \bar z_j|
\label{BergmanMetO}
\end{align}
and $\d g_\fin=\otimes_{p<\infty}\d g_p$ is the product measure of Haar measures $\d g_p$ on $\sG(\Q_p)$ so normalized that $\vol(\bK_p)=1$. Let $\bG=\sG^{0}$ be the special orthogonal group of $(V,Q)$. Then, for each prime number $p$, $\bG(\Z_p)=\bG(\Q_p)\cap \bK_p$ is a maximal compact subgroup of $\bG(\Q_p)$ stabilizing the lattice $\cL_p$. Since $\dim(V)=5$ is odd, $\sG$ is the direct product of $\bG$ and $\{\pm 1_5\}$, the center of $\sG$. Thus by restricting functions to $\bG(\Q_p)$ we obtain an isomorphism $\cH(\sG(\Q_p)\sslash \bG(\Z_p)) \cong \cH(\sG(\Q_p)\sslash \bK_p)$. For $\nu=(\nu_1,\nu_2)\in \fX_p$, let $I_p^{\bG}(\nu)$ denote the minimal principal series of $\bG(\Q_p)$ induced from the unramified character $\chi^{\bG}_{\nu_1,\nu_2}$ of the upper-triangular Borel subgroup ${\mathbb B}(\Q_p)$ of $\bG(\Q_p)$ such that 
\begin{align}
\chi^{\bG}_{\nu_1,\nu_2}:\,\diag(t_1,t_2,1,t_2^{-1},t_1^{-1}) \rightarrow |t|_p^{\nu_1}|t_2|_p^{\nu_2}.
 \label{SOBorelchar}
\end{align}
Let $\pi_{p}^{\bG}(\nu)$ be the unique $\bG(\Z_p)$-spherical constituent of $I_p^{\bG}(\nu)$. For each ${\rm F}\in \cF_l$, let $\{(\alpha_p,\beta_p)\}_{p<\infty}$ be the set of Satake parameters of ${\rm F}$, i.e., for each $p<\infty$, the spherical function corresponding to the eigencharacter $\lambda_{{\rm F},p}:\cH(\bG(\Q_p)\sslash \bG(\Z_p))\rightarrow \C$ on ${\rm F}$ is obtained from the $\bG(\Z_p)$-invariant vector in $\pi_{p}^{\bG}(\nu)$, where $\nu=(\nu_{1,p},\nu_{2,p}) \in \fX_p$ is determined by $\alpha_p=p^{-\nu_{1,p}}$, $\beta_p=p^{-\nu_{2,p}}$. The local $p$-factor of $\lambda_{{\rm F},p}$ is then defined as
$$
L_{p}(s,\lambda_{{\rm F},p})=(1-\alpha_p p^{-s})^{-1}(1-\beta_p p^{-s})^{-1}(1-\alpha_p^{-1}p^{-s})^{-1}(1-\beta_p^{-1}p^{-s})^{-1}.
$$
Then the standard $L$-function of ${\rm F}$ is defined
as the degree $4$ Euler product
$$
L_\fin({\rm F},s)=\prod_{p<\infty} L(s,\lambda_{{\rm F},p}), 
$$
which is shown to be absolutely convergent on $\Re\,s>4$. The completed $L$-function
$$
L({\rm F},s)=\Gamma_{\C}(s+1)\Gamma_{\C}(s+l-3/2)\,L_\fin({\rm F},s)
$$
is continued to a meromorphic function on $\C$ which is holomorphic except for possible simple poles at $s=3/2$ and $s=-1/2$ satisfying the functional equation $$
L({\rm F},1-s)=L({\rm F},s).
$$
For a finite set $S$ of prime numbers and ${\rm F}\in \cF_l$, set 
$$\nu_{S}({\rm F}):=\{(\nu_{1,p},\nu_{2,p})\}_{p\in S} \in \fX_S:=\prod_{p\in S}(\C/2\pi \sqrt{-1}(\log p)^{-1}\Z)^2. 
$$
Let $D<0$ be a fundamental discriminant. Let $\cV(\xi_D;\bK_{1,\fin}^{\xi_D*})$ be the space of all the smooth $\C$-valued functions $f$ on $\sG_1^{\xi_D}(\Q)\bsl \sG_1^{\xi_D}(\A)$ such that $f(h u_\infty u_\fin)=f(h)$ for all $u_\infty\in \sG_1^{\xi_D}(\R)$, $u_\fin \in \bK_{1,\fin}^{\xi_D *}$. We endow the group $\sG_1^{\xi_D}(\A)$ with a Haar measure $\d h=\otimes_{p\leq \infty}\d h_p$, where $\d h_\infty$ is the probability Haar measure on the compact group $\sG_1^{\xi_D}(\R)$ and the measure $\d h_p$ on $\sG_1^{\xi_D}(\Q_p)$ with $p<\infty$ is so normalized that $\vol(\bK_{1,p}^{\xi_D*})=1$. Let $f \in \cV(\xi_{D};\bK_{1,\fin}^{\xi_D *})$ be a simultaneous eigenfunction of the Hecke algebra $\cH^{+}(\sG_1^{\xi_D}(\A_\fin)\sslash \bK_{1,\fin}^{\xi_D *})$. Then set
\begin{align*}
a_{{\rm F}}^{f}(D)&=\sum_{j\in \cJ}f_{\chi}(\tilde u_j)a_{{\rm F}}(\tilde u_j;\xi_D)/e_j, \\
\fa_{{\rm F}}^{f}(D)&=(4\pi\sqrt{2|Q_1(\xi_D)|})^{3/2-l}\,\Gamma(2l-1)^{1/2}\,a_{{\rm F}}^{f}(D),
\end{align*}
where $\{\tilde u_j\}_{j\in \cJ}$ and $e_j\,(j\in  \cJ)$ are as in Lemma~\ref{GxiD-class}, and denote by $\|f\|_{\sG_1^{\xi_D}}$ the $L^2$-norm of $f$ viewed as an element of $L^2(\sG_1^{\xi_D}(\Q)\bsl \sG_1^{\xi_D}(\A),\d h)$. Let $\cU$ be an irreducible $\sG_1^{\xi_D}(\A_\fin)$-submodule of $\cV(\xi_D)$ containing $\bK_{1,\fin}^{\xi_D*}$-fixed vectors, and $L_{\fin}(s,\cU)$ be the standard $L$-function of $\cU$ defined in \cite{MS98}. The completed $L$-function $L(s,\cU)=\Gamma_{\cU}(s)\,L_\fin(s,\cU)$ with $\Gamma_{\cU}(s):=(2\pi)^{-s}\Gamma(s)D^{s/2}$ satisfies the functional equation $L(1-s,\cU)=L(s,\cU)$ (\cite[Theorem]{MS98} and \cite[\S13.2]{Tsud2019}). For a finite set $S$ of prime numbers such that $2\not\in S$ and $p\not \in S$ for all prime divisors $p|D$, let $\fX_S^{+0}=\prod_{p\in S}\fX_{p}^{+0}$ and $W(C_2)^{S}=\prod_{p\in S}W(C_2)$, where $\fX_{p}^{+0}$ is the set of $\nu \in \fX_p$ such that $\pi_{p}^{\bG}(\nu)$ is unitarizable and we consider the coordinate-wise action of $W(C_2)^{S}$ on $\fX_{S}^{+0}$. Let ${\bf \Lambda}^{\xi_D,z_S}(s)=\bigotimes_{p \in S}\Lambda_{p}^{\xi_D,\cU}(s)$ with $s\in \fX_p$ be the Radon measure on the space $\fX_S^{+0}/W(C_2)^{S}$ defined by the formula \cite[(5.20)]{Tsud2019}, or explicitly given by \eqref{LambdaXis} below. Let $\cB(\cU;\bK_{1,\fin}^{\xi_D*})$ be an orthonormal basis of $\cU\cap \cV(\xi_D;\bK_{1,\fin}^{\xi_D*})$. Let $D_{*}(s)$ be the polynomial function of $s$ defined in \cite[\S 2.12]{Tsud2019}, or explicitly $D_*(s)=s^2-1$ in our case. Then \cite[Theorem 1.1 and Theorem 1.2]{Tsud2019} yields the following.  

\begin{thm} \label{th1.1-Tsud2019} Let $\phi=\otimes_{p<\infty}\phi_p \in \cH(\sG(\A_\fin)\sslash \bK_\fin)$ be any Hecke function such that $\phi_p=1_{\bK_p}$ for $p\not\in S$, where $S$ is a finite set of odd prime numbers. Then there exists a constant $C=C(\phi,D)>1$ such that as $l\in 2\N$ grows to infinity, 
\begin{align*}
&\frac{{\bf \Gamma}(l)}{4l^3}
\sum_{{\rm F}\in \cF_l }
\widehat{\phi_S}(\nu_S({\rm F}))\,\frac{L_{\fin}(1/2,{\rm F})}{\langle {\rm F},{\rm F}\rangle}\,
\sum_{f\in \cB(\cU.\bK_{1,\fin}^{\xi_D*})} {|\fa_{{\rm F}}^{f}(D)|^2} \\
&= 4\left(\tfrac{\pi}{4}\right)^{-1}\,\biggl\{
{\bf \Lambda}^{\xi_D,\cU}(0;\widehat{\phi_S})\,{\rm Res}_{s=1}L_\fin(s,\cU)\,\biggl(\psi(l-1)+\frac{\Gamma_{\cU}'(1)}{\Gamma_\cU(1)}-\frac{D_{*}'(0)}{D_*(0)}-\log(\sqrt{8|Q_1(\xi_D)|}\pi)\biggr) \\
&\quad +{\rm Res}_{s=1}L_\fin(s,\cU)\,\biggl(\tfrac{\d}{\d s}|_{s=0}{\bf \Lambda}^{\xi_D,\cU}(s;{\widehat \phi_S})\biggr)
+{\bf \Lambda}^{\xi_D,\cU}(0;{\widehat \phi_S})\,{\rm CT}_{s=1}L_\fin(s,\cU)\biggr\}+O(C^{-l}),
\end{align*}
where $${\bf \Gamma}(l)=\frac{l^3\Gamma(l-3/2)\Gamma(l-2)}{\Gamma(l-1/2)\Gamma(l)}.$$
\end{thm}
\begin{proof} From Lemma~\ref{L:IrredDecomp-cV}, we may suppose $\cU=\cU_\chi$ with some $\chi \in \widehat{{\rm Cl}_D}$. We apply \cite[Theorem 1.1, Theorem 1.2]{Tsud2019} to our $(V,Q)$ taking $\xi=\xi_D$ and $\cU=\cU_{\chi}$. We have $m=3$ and $\rho=(3-1)/2=1$. Moreover, from Lemmas~\ref{L:tauxiD} and \ref{L:IrredDecomp-cV}, $\cU_\chi(\bK_{1,\fin}^{\xi_D*})=\C f_\chi$, $d^{+}(\cU_\chi)=1$, $d^{-}(\cU_\chi)=0$ and $\chi(\cU_\chi)=1$. Note that $\#\cB(\cU_\chi;\bK_{1,\fin}^{\xi_D*})=1$. Although \cite[Theorem 1.2]{Tsud2019} only describes the main term of the asymptotic formula, the argument to prove \cite[Proposition 5.9]{Tsud2019} is easily extended to the case when $L_\fin(s,\cU)$ has a pole at $s=1$. 
\end{proof}

To simplify the formula further, we use the following lemma. 
\begin{lem} \label{L:BesselPer}
 Let $\chi \in \widehat{{\rm Cl}_D}$ and $f=f_\chi \in \cV(\xi_D,\bK_{1,\fin}^{\xi_D*})$ be the Hecke eigen function defined by \eqref{def:fchi}. Then  
\begin{align*}
\frac{1}{4l^3} \, 
{|\fa_{{\rm F}}^{f}(D)|^2}
=2\pi^{-1}\,\left(1-\tfrac{3}{2l}\right)\left(1-\tfrac{2}{l}\right)\left(1-\tfrac{1}{l}\right)\, \left(\tfrac{|D|}{4} \right)^{3/2-l}\,c_{l}
\,w_D^{-2}\,|{\rm R}({\rm F},D,\chi)|^2 
\end{align*}
with 
$$
c_l=\frac{\sqrt{\pi}}{4}(4\pi)^{3-2l}\Gamma(l-3/2)\Gamma(l-2), \qquad 
{\rm R}({\rm F},D, \chi):=\sum_{j=1}^{h_D}a_{{\rm F}}(1;T_j w)\chi(c_j), $$
where $\{T_j\}_{j=1}^{h_D}$ is a complete set of representatives in ${\bf SL}_2(\Z)\bsl \cQ_{{\rm prim}}^{+}(D)$ and $c_j\in\A_{E,\fin}^{\times}/E^\times \widehat{\cO_E}^{\times} $ is the image of $T_j$ under the map ${\bf SL}_2(\Z)\bsl \cQ_{\rm prim}^{+}(D)\rightarrow \A_{E,\fin}^\times/E^\times \widehat{\cO_E}^\times$ obtained by Lemmas~\ref{L-GxiDclassBiQ} and \ref{L-CLd-stab}. 
\end{lem}
\begin{proof} Recall some material from \cite[\S 2.11]{Tsud2019}. Set $F(g_\fin g_\infty)=J(g_\infty, \fz_0)^{-l}\,{\rm F}(g_\infty\langle \fz_0 \rangle;g_\fin)$ for $g_\fin \in \sG(\A_\fin)$ and $g_\infty \in \sG(\R)^+$. For $\eta \in V_1(\R)$ such that $Q_1(\eta)<0$, let 
$$\cW_l^{\eta}(g_\infty)=J(g_\infty,\fz_0)^{-l}\exp(2\pi \sqrt{-1} Q_1(\eta,g_\infty\langle \fz_0 \rangle)), \quad g_\infty\in \sG(\R)^{0}
$$
be the holomorphic archimedean Whittaker function of weight $l$. Then, 
$$
a_{{\rm F}}(g_\fin;\eta)\,\cW_l^{\eta}(g_\infty)
=\int_{V_1(\Q)\bsl V_1(\A)} F(\sn(X)g_\fin g_\infty)\,\psi(-Q_1(\eta,X))\,\d X,
$$
where $\psi:\Q\bsl \A\rightarrow \C$ is a character determined by $\psi(x)=e^{2\pi \sqrt{-1}x}\,(x\in \R)$, 
\begin{align*}\sn(X)=\left[\begin{smallmatrix} 1 & {-{}^tX Q_1} & {-2^{-1}Q_1[X]} \\ {0} & {1_3} & {X} \\ {0} & {0} & {1} \end{smallmatrix} \right]
\end{align*}
and $\d X$ is the Haar measure on $V_1(\A)$ such that $\vol(V_1(\Q)\bsl V_1(\A))=1$. Let $\{\tilde u_j\}_{j \in \cJ}$ be as in Lemma~\ref{GxiD-class}; for each $j\in \cJ$, choose $\gamma_j\in \sG_1(\Q)$, $h_{j}\in \sG_1(\R)$, and $\kappa_j\in \bK_{1,\fin}$ such that $\tilde u_j=\gamma_{j}h_j\kappa_j$. Then by the construction of the bijection $$\A_{E,\fin}^\times/E^\times \widehat{\cO_E}^\times \cong \sG_{1}^{\xi_D}(\Q)\bsl \sG_1^{\xi_D}(\A_\fin)/\bK_{1,\fin}^{\xi_D*} \cong {\bf SL}_2(\Z)\bsl \cQ_{{\rm prim}}^{+}(D),$$
 we see that $\tilde u_j \in \sG_1^{\xi_D}(\Q)\bsl \sG_{1}^{\xi_D}(\A_\fin)/\bK_{1,\fin}^{\xi_D*}$ and $c_j \in \A_{E,\fin}^{\times}/E^\times{\widehat \cO_E}^\times$ and the class of $T_j:=(\gamma_j^{-1}\cdot \xi_D) w^{-1}=(\det\gamma_j)\,{}^t\gamma_j^{-1}\xi_Dw^{-1} \gamma_j^{-1}$ correspond to each other. For $h\in \sG_1(\Q)$, let $\sm(h)=\diag(1,h,1)$ be its image in $\sG(\A)$. Since $F$ is left $\sG(\Q)$-invariant and right $\bK_{\fin}$-invariant, 
{\allowdisplaybreaks\begin{align*}
a_{{\rm F}}^{f}(D)\,\cW_l^{\xi_D}(g_\infty)&=\sum_{j\in \cJ}f(\tilde u_j)\int_{V_1(\Q)\bsl V_1(\A)} F(\sn(X)\sm(\tilde u_j)\,g_\infty)\,\psi(-Q_1(\xi_D,X))\,\d X
\\
&=\sum_{j\in \cJ}f(\tilde u_j)\int_{V_1(\Q)\bsl V_1(\A)} F(\sn(X)\sm(\gamma_j h_j \kappa_j)\,g_\infty)\,\psi(-Q_1(\xi_D,X))\,\d X
\\
&=\sum_{j\in \cJ}f(\tilde u_j)\int_{V_1(\Q)\bsl V_1(\A)} F(\sn(\gamma_j^{-1}\,X)\,\sm(h_j)\,g_\infty)\,\psi(-Q_1(\xi_D,X))\,\d X
\\
&=\sum_{j\in \cJ}f(\tilde u_j)\int_{V_1(\Q)\bsl V_1(\A)} F(\sn(X)\,\sm(h_j)\,g_\infty)\,\psi(-Q_1(\gamma_j^{-1} \xi_D,X))\,\d X
\\
&=\sum_{j\in \cJ}f(\tilde u_j)\,a_{{\rm F}}(1;\gamma_j^{-1}\xi_D)\cW_l^{\gamma_j^{-1}\xi_D}(\sm(h_j)g_\infty). 
\end{align*}}Noting that $g\mapsto J(g_\infty,\fz)$ is left $\sG_1(\R)$-invariant and the image of $\gamma_j^{-1}$ in $\sG_(\R)$ equals $h_j$, we easily confirm $\cW_l^{\gamma_j^{-1}\xi_D}(\sm(h_j)g_\infty)=\cW_l^{\xi_D}(g_\infty)$. Thus we obtain the expression:
$$
a_{{\rm F}}^{f}(D)=\sum_{j\in \cJ}f(\tilde u_j)\,a_{{\rm F}}(1;T_jw). 
$$
Set $\cJ_1=\{j\in \cJ|\,\hat{j}=j\}$ and $\cJ_2=\{j\in \cJ|\hat{j}\not=j\}$, where $j\mapsto \hat {j}$ is as in Lemma~\ref{GxiD-class}. For $u\in {\rm Cl}_D$, let $[u]$ denote the ${\rm Gal}(E/\Q)$-orbit of $u$. Then $[u_j]=\{u_j\}$ if $j\in \cJ_1$ and $[u_j]=\{u_j,\bar u_j\}$ if $j\in \cJ_2$. 
Since $\ss\circ \iota(\bar t)=\tilde \sigma\, (\ss\circ \iota(t))\,\tilde \sigma$ for $t\in \A_{E,\fin}^{\times}$ and $\tilde \sigma\in \sG_1^{\xi_D}(\Q)\cap \bK_{1,\fin}^{\xi_*}$, we may suppose $\gamma_{\hat j}=\tilde \sigma \gamma_{j}$ and thus $\gamma_{\hat j}^{-1}\xi_D=\gamma_{j}^{-1}\xi_D$. From Lemma~\ref{GxiD-class}, $e_j=w_D$ if $j\in \cJ_1$ and $e_j=w_D/2$ if $j\in \cJ_2$. Hence
\begin{align*}
a_{{\rm F}}^{f_\chi}(D)&=\frac{1}{w_D}
\sum_{j \in \cJ_1} \tfrac{1}{2}(\chi(u_j)+\chi(\bar u_j))\,a_{{\rm F}}(1;\gamma_j^{-1}\xi_D)
+\frac{2}{w_D}
\sum_{j \in \cJ_2} \tfrac{1}{2}(\chi(u_j)+\chi(\bar u_j))\,a_{{\rm F}}(;\gamma_j^{-1}\xi_D)
\\
&=\frac{1}{w_D} \sum_{j\in \cJ_1} \sum_{u\in [u_j]}\chi(u)\,a_{{\rm F}}(1;\gamma_j^{-1}\xi_D)+\frac{1}{w_D}\sum_{j\in \cJ_2}\sum_{u \in [u_j]}
\chi(u)\,a_{{\rm F}}(1;\gamma_j^{-1}\xi_D)\\
&=\frac{1}{w_D}\sum_{j\in \cJ}\sum_{u\in [u_j]}\chi(u)\,a_{{\rm F}}(1;\gamma_j^{-1} \xi_D)
=\frac{1}{w_D}\sum_{j=1}^{h_D} \chi(c_j)a_{{\rm F}}(1;T_jw).
\end{align*}
Since $Q_1(\xi_D)=D/2$, by the duplication formula of the gamma function, we have {\allowdisplaybreaks\begin{align*}
&\frac{1}{4l^3}\,\Bigl\{(4\pi\sqrt{2|Q_1(\xi_D)|})^{3/2-l}\,\Gamma(2l-1)^{1/2}\Bigr\}^2 \\
&=\frac{1}{4l^3}(4\pi)^{3-2l}\left(\tfrac{|D|}{4}\right)^{3/2-l}\times 4^{3/2-l}\times (2^{2l-2}\pi^{-1/2}\Gamma\left(l-\tfrac{1}{2}\right)\Gamma(l))
\\
&=\frac{2\pi^{-1}}{l^3}(l-3/2)(l-2)(l-1)\times \tfrac{\sqrt{\pi}}{4}(4\pi)^{3-2l}\left(\tfrac{|D|}{4}\right)^{3/2-l}\Gamma(l-3/2)\Gamma(l-2)
\\
&=2\pi^{-1}\,\left(1-\tfrac{3}{2l}\right)\left(1-\tfrac{2}{l}\right)\left(1-\tfrac{1}{l}\right)\, \left(\tfrac{|D|}{4} \right)^{3/2-l}\,c_{l}.
\end{align*}}\end{proof}

Let $S$ be a finite set of prime numbers. For $\cU=\cU_{\chi}$ and $s\in \C$, the measure ${\bf \Lambda}^{\xi_D,\cU}(s)$, denoted by ${\bf \Lambda}^{\xi_D,\chi}(s)$, is given by
\begin{align}
{\bf \Lambda}^{\xi_D,\chi}(s) =\bigotimes_{p\in S}\, \frac{\zeta_p(2)\,\zeta_p(4)}{\zeta_{p}(1)\,L(s+1,{\rm AI}(\chi)_p)} \,
\frac{L\left(\frac{1}{2},\pi_{p}^{\bG}(\nu)\times {\rm AI}(\chi)_p \right)L\left(\frac{1}{2}+s,\pi_p^{\bG}(\nu)\right)}{L(1,\pi_p^{\bG}(\nu),{\rm Ad})}\,\d \mu_{p}^{{\rm Pl}}(\nu),
 \label{LambdaXis}
\end{align}
where $\d \mu_{p}^{{\rm Pl}}(\nu)$ is the spherical Plancherel measure describing the spectral decomposition of $L^2(\bG(\Q_p)/\bG(\Z_p),\d g_p)$. Set ${\bf \Lambda}^{\xi_D,\chi}:={\bf \Lambda}^{\xi_D,\chi}(0)$.  

\begin{cor} \label{cor:Oasympt}
Let $\chi$ be a character of ${\rm Cl}_D=\A_{E,\fin}^\times /E^\times \widehat {\cO_E}^\times$. Let $S$ be a finite set of odd prime numbers such that $p\not\in S$ for all prime divisors $p|D$. Let $\phi=\otimes_p \phi_p$ is any Hecke function such that $phi_p=1_{\bK_p}$ for all $p\not\in S$. Then as $l\in 2\N$ grows to infinity, we have 
\begin{align*}
{d_\chi\,c_{l,D}}
\sum_{{\rm F}\in \cF_l}
\widehat{\phi_S}(\nu_S({\rm F}))\,{L_{\fin}(1/2,{\rm F})}\,\frac{|{\rm R}({\rm F},D,\chi)|^2}{ \langle {\rm F},{\rm F}\rangle}=32\,P(l,D,\chi;\widehat{\phi_S})+O(C^{-l}),
\end{align*}
where $P(l,D,\chi;{\widehat{\phi_S}})$ is equal to $ L_\fin(1,{\rm AI}(\chi))\,{\bf \Lambda}^{\xi_D,\chi}(\widehat{\phi_S})$ if $\chi\not={\bf 1}$, and to 
$$
\biggl(L_\fin(1,\eta_D)\,(\psi(l-1)-\log(4\pi^2))+L_\fin'(1,\eta_D)\biggr)\,{\bf \Lambda}^{\xi_D,{\bf 1}}(\widehat{\phi_S})+L'_{\fin}(1,\eta_D)\,\left(\tfrac{\d}{\d s}|_{s=0}{\bf \Lambda}^{\xi_D,{\bf 1}}(s;\widehat{\phi_S})\right)
$$
 if $\chi={\bf 1}$. 
\end{cor}
\begin{proof}
This follows from Theorem~\ref{th1.1-Tsud2019}, Lemma~\ref{L:BesselPer} and Lemma~\ref{L:IrredDecomp-cV}. To simplify the formula when $\chi={\bf 1}$, we note the relations $L_\fin(1,\cU_{\chi})=\zeta(s)L_\fin(s,\eta_D)$, 
\begin{align*}
\frac{\Gamma_{\cU}'(1)}{\Gamma_\cU(1)}&=\tfrac{1}{2}\log|D|-\log(2\pi)+\psi(1), &
\frac{D_*(0)}{D_*(0)}&=0, \\
{\rm Res}_{s=1}L_\fin(s,\cU_\chi)&=L_\fin(1,\eta_D), &
{\rm CT}_{s=1}L_\fin(s,\cU_\chi)&=L'_\fin(1,\eta_D)+\gamma_0\,L_\fin(1,\eta_D),
\end{align*}
where $\gamma_0=-\psi(1)$ is the Euler-Mascheroni constant. From these, we easily have the equality
\begin{align*}
&{\rm Res}_{s=1}L_\fin(s,\cU_\chi)\,
\biggl(\psi(l-1)+\frac{\Gamma_{\cU}'(1)}{\Gamma_\cU(1)}-\frac{D_{*}'(0)}{D_*(0)}-\log(\sqrt{8|Q_1(\xi_D)|}\pi) \biggr)
+{\rm CT}_{s=1}L_{\fin}(s,\cU_\chi)
\\
&=(\psi(l-1)-\log(4\pi^2))L_\fin(1,\eta_D)+L'_\fin(1,\eta_D). 
\end{align*}
We also note the relation
$$
\left(1-\tfrac{3}{2l}\right)\left(1-\tfrac{2}{l}\right)\left(1-\tfrac{1}{l}\right)\times {\bf \Gamma}(l)=1.$$
\end{proof}

\section{Proof of main result} 
Recall the notation for Siegel modular forms and ${\bf G}={\bf PGSp}_2$ introduced in \S~\ref{sec:Intro}. As is well-known, there is an exceptional isomorphism ${\bf G}\cong {\bf SO}(Q)$ which yields a linear isomorphism between the spaces of modular forms $S_{l}({\bf Sp}_2(\Z))$ and $S_{l}(\bK_\fin)$ preserving $L$-functions and periods (for a precise statement, see Proposition~\ref{SpO-Isom}), which allows us to transcribe Corollary~\ref{cor:Oasympt} in the language of Siegel modular forms. If we take $S$ to be the empty set, then we obtain Theorem~\ref{thm0} from Corollary~\ref{cor:Oasympt}. In the remaining part of this section, we only focus on the main terms; noting the asymptotic formulas $\tilde {\bf \Gamma}(l)=1+O(l^{-1})$ and $\psi(l-1)=\log l+O(l^{-1})$, we have the following proposition from Corollary~\ref{cor:Oasympt}. 

\begin{prop} \label{GSPasympt}
Let $\chi$ be a character of ${\rm Cl}_D=\A_{E,\fin}^\times /E^\times \widehat {\cO_E}^\times$. Let $S$ be a finite set of odd prime numbers such that $p\not\in S$ for all prime divisors $p|D$. Let $\phi=\otimes_p \phi_p \in \cH({\bf G}(\A_\fin)\sslash {\bf G}(\widehat \Z))$ is any Hecke function such that $\phi_p=1_{{\bf G}_(\Z_p)}$ for all $p\not\in S$. Then as $l\in 2\N$ grows to infinity, we have \begin{align*}
\frac{1}{(\log l)^{\delta(\chi={\bf 1})}}
\sum_{\Phi \in \cF_l}
\widehat{\phi_S}(\nu_S(\Phi))\,{L_{\fin}(1/2,\pi_\Phi)}\,\omega_{l,D,\chi}^{\Phi}\rightarrow  2{\bf \Lambda}^{\chi} (\widehat{\phi_S})\, 
\begin{cases} 
L_\fin(1,\eta_D), \quad(\chi=1),\\
L_\fin(1,{\mathcal{AI}}(\chi)),  \quad (\chi\not=1).
\end{cases}
\end{align*}
\end{prop}
If the non-negativity of the central values $L_\fin(1/2,\pi_\Phi)$ were available, we would obtain the limit formula in Theorem~\ref{MAINTHM} for the average over $\cF_l$ directly from this by a familiar approximation argument ({\it cf}. \cite{Serre}). But this expectation seems to be very hard to be realized, due to the existence of CAP forms. Let $\Phi={\rm SK}(f) \in S_{l}({\bf Sp}_2(\Z))$ be the Saito-Kurokawa lifting from an elliptic Hecke-eigen cuspform $f$ on ${\bf SL}_2(\Z)$ of weight $2l-2$. Then the following formula is well known.
$$
L_{\fin}(s,\pi_\Phi)=\zeta(s+1/2)\zeta(s-1/2)L_\fin(s,f).
$$
Since the sign of the functional equation of $f$ is minus, $L_\fin(1/2,f)=0$. Noting this, we obtain 
$$
L_{\fin}(1/2,\pi_\Phi)=\zeta(0)\,L'_{\fin}(1/2,f).
$$
At present, our knowledge on the sign of this quantity is very restrictive. However, concerning the size of this, the trivial bound $|L_{\fin}'(1/2,f)|\ll_{\e} l^{1/2+\e}$ immediately gives us 
\begin{align}
|L_{\fin}(1/2,\pi_{{\rm SK}(f)})|\ll_{\e} l^{1/2+\epsilon}, \quad f\in \cH_{2l-2},
 \label{TrivBoundL}
\end{align}
where $\cH_{2l-2}$ is the set of the normalized Heck eigen elliptic cuspforms on ${\bf SL}_2(\Z)$ of weight $2l-2$. From \cite[\S 5.3]{KST}, we quote the following formula for $\Phi={\rm SK}(f)$. $$
\omega_{l,D,\chi}^{\Phi}=\delta(\chi={\bf 1})\,\frac{(48\pi)^2h_{D}}{w_{D}(l-1)(l-2)}\frac{\Gamma(2l-3)}{(4\pi)^{2k-3}\langle f, f\rangle}\frac{L_\fin(1/2,f\times \eta_{D})}{L_\fin(1,f)}.
$$
To prove Theorem~\ref{MAINTHM}, we follow the same strategy employed by \cite{KST} and \cite{KYW}. Indeed, we showed in \cite[\S5.2]{Tsud2019} that the argument works for a general orthogonal group conditionally on two hypothesis \cite[(1.7) and (1.8)]{Tsud2019}. For our $(V,Q)$, due to the deep results on automorphic representations of ${\bf GSp}_2$, we can make the argument unconditional. First, the following lemma, which is a direct consequence of \cite[Proposition 5.8]{KST} and \eqref{TrivBoundL}, implies the statement \cite[(1.8)]{Tsud2019} is true. 
\begin{lem} \label{SKnegligible}
 As $l\in 2\N$ grows to infinity, 
\begin{align*}
\frac{1}{(\log l)^{\delta(\chi={\bf 1})}}\,\sum_{f\in \cH_{2l-2}} |L_\fin(1/2,\pi_{{\rm SK}(f)})|\,\omega_{l,D,\chi^{-1}}^{{\rm SK}(f)}\longrightarrow 0. 
\end{align*} 
\end{lem}
This lemma also imples the second limit formula in Theorem~\ref{MAINTHM}. The truth of the statement \cite[(1.7)]{Tsud2019}, which boils down to the statement
\begin{align}
\text{
$L_{\fin}(1/2,\pi_\Phi)\geq 0$ for all $\Phi \in \cF_l^{\flat}$. }
 \label{NonNegative}
\end{align}
is known by \cite[Theorem 5.2.4]{PSS}. Thus we see that \cite[(1.7) and (1.8)]{Tsud2019} are satisfied. Starting from Proposition~\ref{GSPasympt}, by the same argument as in \cite[\S5.2]{Tsud2019}, we complete the proof of Theorem~\ref{MAINTHM}. \qed 

Since $L_{\fin}(1,\eta_D)\not=0$ and $L_{\fin}(1,{\mathcal{AI}}(\chi))\not=0$ if $\chi\not={\bf 1}$, Corollary~\ref{MAINTHM2} is obtained from Theorem~\ref{MAINTHM} by approximating the characteristic function by a continuous function.

\subsection{Book-keeping for exceptional isomorphism} \label{sect:Bookkeep}
For convenience, we collect miscellaneous facts which is useful to derive Proposition~\ref{GSPasympt} from Corollary~\ref{cor:Oasympt}. For our purpose, it is convenient to use the $5$-dimensional quadratic space
$$
V=\left\{Y=\left[\begin{smallmatrix} X & -x' w \\ x'' w & {}^t X\end{smallmatrix} \right]|\,X\in V_1,\,x',x'' \in \Q \right\}
$$
over $\Q$ with the quadratic from $q(Y)=\frac{1}{2}\det(Y^2)$ (\cite[\S 6.3]{Liu}). By letting the group ${\bf GSp}_2$ act on $V$ as $\rho(g)Y=gYg^{-1}$, we have a surjective $\Q$-morphism $\rho:{\bf GSp}_2\rightarrow {\rm SO}(V)$ whose kernel coincides with the center of ${\bf GSp}_2$. Thus $\rho$ realizes the exceptional isomorphism ${\bf PGSp}_2 \cong {\rm SO}(V)$. Set 
\begin{align*}
\e_0=\left[\begin{smallmatrix} 
{} & {\begin{smallmatrix} 0 & 1 \\ -1 & 0 \end{smallmatrix}} 
\\ {\begin{smallmatrix} {0} & {0}\\ {0} & {0} \end{smallmatrix}} & {} \end{smallmatrix}\right], 
\e_0'=\left[\begin{smallmatrix} {} & {\begin{smallmatrix} 0 & 0 \\ 0 & 0 \end{smallmatrix}} \\ \begin{smallmatrix} 0 & -1 \\ 1 & 0 \end{smallmatrix} & {} \end{smallmatrix} \right],
\e_1=\left[\begin{smallmatrix}  \begin{smallmatrix} 0 & 1 \\ 0 & 0 \end{smallmatrix} & {} \\ {} & \begin{smallmatrix} 0 & 0 \\ 1 & 0 \end{smallmatrix} \end{smallmatrix} \right]
\e_1'=\left[\begin{smallmatrix}  \begin{smallmatrix} 0 & 0 \\ 1 & 0 \end{smallmatrix} & {} \\ {} & \begin{smallmatrix} 0 & 1 \\ 0 & 0 \end{smallmatrix} \end{smallmatrix} \right], 
v=\left[\begin{smallmatrix}  \begin{smallmatrix} 1 & {0} \\ 0 & -1 \end{smallmatrix} & {} \\ {} & \begin{smallmatrix} 1 & 0 \\ 0 & -1 \end{smallmatrix} \end{smallmatrix} \right].
\end{align*}
Then these vectors form a $\Q$-basis of $V$ such that 
$$
q(x_1\e_1+x_0\e_0+zv+y_0\e_0'+y_1\e_1')=
[x_1,x_0,z,y_0,y_1]\, Q\, 
\,\left[\begin{smallmatrix} x_1 \\ x_0 \\ z \\ y_0 \\ y_1\end{smallmatrix}\right]
$$
with $Q$ given by \eqref{matrixQ}. We use the matrix realization of ${\rm O}(V)$ as ${\bf O}(Q)$ identifying an element $\tilde h\in {\rm O}(V)$ with the matrix $h\in {\bf O}(Q)$ determined by the relation
$$
[\tilde h(\e_1),\tilde h(\e_0),\tilde h(v), \tilde h(\e_0'),\tilde h(\e_1')]
=[\e_1,\e_0,v,\e_0',\e_1']\,h.  
$$
The particular elements $\sn(X)$ for $X\in \Q^3$ and $\sm(t;h)$ for $t\in \Q^\times$, $h\in \sG_1:={\bf O}(Q_1)$ of the matrix group $\sG:={\bf O}(Q)$ is defined as 
\begin{align*}
\sm(r;h)&={\rm diag}(r,h,r^{-1}), \quad \sn(X)=\left[\begin{smallmatrix} 1 & {-{}^tX Q_1} & {-2^{-1}Q_1[X]} \\ {0} & {1_3} & {X} \\ {0} & {0} & {1} \end{smallmatrix} \right].
\end{align*}
Then for $\rho:{\bf GSp}_2 \rightarrow {\bf SO}(Q)$, the following formula is easily confirmed
\begin{align}
\rho\left(\left[\begin{smallmatrix} 1_2 & B \\ {} & 1_2 \end{smallmatrix} \right] \right)&=\sn(\left[\begin{smallmatrix} b_1 \\ -b_2 \\ -b_3 \end{smallmatrix} \right]), \quad B=\left[\begin{smallmatrix} b_1 & b_2 \\ b_2 & b_3 \end{smallmatrix}\right], 
 \label{rhoOnSP1}
\\
\rho\left(\left[\begin{smallmatrix} A &  \\ {} & \nu {}^t A^{-1} \end{smallmatrix} \right] \right)
&=\sm(\nu^{-1} \det(A); \ss(A)), \quad A\in {\bf GL}_2,\,\nu \in {\bf GL}_1,
 \label{rhoOnSP2}
\end{align}
where $\ss(A)$ is the matrix given by \eqref{ssFormula}. From thses, the Siegel parabolic subgroup of ${\bf GSp}_2$ corresponds to the maximal parabolic subgroup $\sP$ stabilizing the line $\Q\e_1$.  

 There exists a unique isomorphism $\jj_{\cD}:\fh_2\rightarrow \cD$ such that $\jj_{\cD}(\sqrt{-1} 1_2)=\fz_0$ and 
\begin{align}
\jj_{\cD}(g.Z)=\rho(g)\langle \jj_{\cD}(Z)\rangle, \quad g\in {\bf GSp}_2(\R),\,Z\in \fh_2.
\label{Rho-RhocD}
\end{align}
Therefore, $\rho$ maps the maximal compact subgroup 
\begin{align}
\left\{\left[\begin{smallmatrix}A & B \\ -B & A \end{smallmatrix}\right]|\,A+\sqrt{-1}B\in U(2)\right\}
 \label{McompSP}
\end{align}
of ${\bf GSp}_2(\R)^0$ onto the maximal compact subgroup $\bK_{\infty}={\rm Stab}_{\sG(\R)^0}(\fz_0)$ of $\sG(\R)^0$.  

\begin{lem} \label{L:jcD}
We have 
$$
\jj_{\cD}(Z)=\left[\begin{smallmatrix} z_1 \\ -z_2 \\ -z_3 \end{smallmatrix} \right] \quad \text{for $Z=\left[\begin{smallmatrix} z_1 & z_2 \\ z_2 & z_3\end{smallmatrix} \right] \in \fh_2$}.
$$
For $g=\left[\begin{smallmatrix} A & B \\ C & D \end{smallmatrix}\right]\in {\bf GSp}_2(\R)^{0}$, we have 
$$
\det(CZ+D)=J(\rho(g), \jj_{\cD}(Z)), \quad Z\in \fh_2. 
$$
\end{lem}
\begin{proof} 
By the Iwasawa decomposition of ${\bf GSp}_2(\R)$, for any element $Z\in \fh_2$ we can find $A=\left[\begin{smallmatrix} a & b \\ c & d\end{smallmatrix}\right] \in {\bf GL}_2(\R)$ and $S=\left[\begin{smallmatrix} s_1 & s_2 \\ s_2 & s_3 \end{smallmatrix}\right]\in {\rm Sym}^{2}(\R)$ such that 
$$
Z=\left[\begin{smallmatrix} A & {} \\ {} & {}^t A^{-1}\end{smallmatrix}\right] \,\left[\begin{smallmatrix} 1_2& {S} \\ {} & 1_2\end{smallmatrix}\right].(\sqrt{-1}1_2).
$$
By a computation, 
\begin{align*}
\sm(\det A; \ss(A))\,\sn(\left[\begin{smallmatrix} s_1 \\ -s_2 \\ -s_3\end{smallmatrix}\right])\left[\begin{smallmatrix} -1 \\ \sqrt{-1} \\ 0 \\ -\sqrt{-1} \\ 1 \end{smallmatrix}\right]
&=\det(A)^{-1} \left[\begin{smallmatrix}* \\ z_1 \\ -z_2 \\ -z_3 \\ 1 \end{smallmatrix} \right], \\
Z=A\,{}^tAi +AS\,{}^t A=\left[\begin{smallmatrix} z_1 & z_2 \\ z_2 & z_3\end{smallmatrix}\right]
\end{align*}
with 
\begin{align*}
z_1&=(a^{2}s_1+2abs_2+b^2 s_3)+\sqrt{-1}(a^2+b^2), \\
z_2&=(ac s_1+(ad+bc)s_2+bds_2)+\sqrt{-1}(ac+bd), \\ 
z_3&=(c^2s_1+2dcs_2+d^2 s_3)+\sqrt{-1}(c^2+d^2). 
\end{align*}
Hence from \eqref{Rho-RhocD}, \eqref{rhoOnSP1}, and \eqref{rhoOnSP2}, 
\begin{align*}
\jj_{\cD}(Z)
&=\rho(\left[\begin{smallmatrix} A & {} \\ {} & {}^t A^{-1}\end{smallmatrix}\right] \left[\begin{smallmatrix} 1_2& {S} \\ {} & 1_2\end{smallmatrix}\right])\,\langle \fz_0 \rangle=\sm(\det A; \ss(A))\,\sn(\left[\begin{smallmatrix} s_1 \\ -s_2 \\ -s_3\end{smallmatrix}\right])\,\langle \fz_0 \rangle=\left[\begin{smallmatrix} z_1 \\ -z_2 \\ -z_3 \end{smallmatrix} \right].
\end{align*}   

Set $J_{\fh_2}(g,Z)=\det(CZ+D)$. We have $J_{\fh_2}(g_1g_2,Z)=J_{\fh_2}(g_1,g_2.Z)J_{\fh_2}(g_2,Z)$ and a similar automorphy relation for $J$. By the Iwasawa decompositions on ${\bf GSp}_2(\R)^0$ and $\sG(\R)^0$, it suffices to show the relation $J_{\fh_2}(g,\sqrt{-1}\,1_2)=J(\rho(g),\fz_0)$ for 
\begin{align}
g=\left[\begin{smallmatrix} A &  \\ {} & \nu\, {}^tA^{-1} \end{smallmatrix} \right]\left[\begin{smallmatrix} 1_2 & B \\ {} & 1_2 \end{smallmatrix} \right]
 \label{SiegelParelements}
\end{align}
and for elements $g$ belonging to \eqref{McompSP}. For $g$ of the form \eqref{SiegelParelements} we easily have $J_{\fh_2}(g,\sqrt{-1}\,1_2)=\nu (\det A)^{-1}$ and $J(\rho(g),\fz_0)=\nu(\det A)^{-1}$ by means of \eqref{rhoOnSP1} and \eqref{rhoOnSP2}. Since $g\mapsto J_{\fh_2}(g,\sqrt{-1}\,1_2)$ and $g\mapsto J(\rho(g),\fz_0)$ are characters of the compact connected group \eqref{McompSP} isomorphic to $U(2)$, it suffices to show 
$$
\tfrac{\d}{\d t}\bigm|_{t=0}J_{\fh_2}(\exp(tH),\sqrt{-1}\,1_2)=\tfrac{\d}{\d t}\bigm|_{t=0} J(\rho(\exp(t H)),\fz_0),
$$
where $H$ is an element in the Lie algebra of \eqref{McompSP} of the form 
$$H=\left[\begin{smallmatrix} {} & {\begin{smallmatrix} x_1 & 0 \\ 0 & x_2 \end{smallmatrix}} \\ {\begin{smallmatrix} -\tau_1 & 0 \\ 0 & -\tau_2 \end{smallmatrix}} & {} \end{smallmatrix} \right]$$
 with $\tau_1,\tau_2\in \R$. By a direct computation,
$$
\d \rho(H)=\left[\begin{smallmatrix} 
0 & \tau_2 & 0 & -\tau_1 & 0 \\
-\tau_2 & 0 &0& 0& \tau_1 \\
0& 0 &0& 0& 0 \\
\tau_1& 0& 0& 0&  -\tau_2 \\
0& -\tau_1 &0 &\tau_2 &0 
\end{smallmatrix} \right]. 
$$
By taking the differential of \eqref{ActionGD} applied to $g=\rho(\exp(tH))$ with $\fz=\fz_0$, we have
\begin{align*}
\d\rho(H)\,\left[\begin{smallmatrix} -1 \\ \fz_0 \\ 1 \end{smallmatrix} \right]
=\tfrac{\d}{\d t}\bigm|_{t=0}J(\rho(\exp(tH),\fz_0)\,\left[\begin{smallmatrix} 
-1\\ \fz_0 \\ 1 \end{smallmatrix}\right].
\end{align*}
with $\fz_0={}^t[\sqrt{-1},0,-\sqrt{-1}]$. Hence 
$$
\tfrac{\d}{\d t}\bigm|_{t=0}J(\rho(\exp(tH),\fz_0)=(-\tau_1,0,\tau_2) \fz_0=-\sqrt{-1}(\tau_1+\tau_2).
$$
On the other hand, from definition 
$$J_{\fh_2}(\exp(tH),i1_2)=\det\left[\begin{smallmatrix} e^{-\sqrt{-1} t\tau_1} & 0 \\ 0 & e^{-\sqrt{-1} t\tau_2} \end{smallmatrix}\right]=e^{-\sqrt{-1}t(\tau_1+\tau_2)}.$$
Hence we have $\tfrac{\d}{\d t}|_{t=0}J_{\fh_2}(\exp(tH),\sqrt{-1}\,1_2)=-\sqrt{-1}(\tau_1+\tau_2)$ as desired.\end{proof}

From Lemma~\ref{L:jcD}, \eqref{BergmannMetSp} and \eqref{BergmanMetO}, we see that the volume forms on $\fh_2$ and on $\cD$ are related by   
\begin{align}
\jj^{*}_{\cD}(\d\mu_{\cD})(Z)=\tfrac{1}{8}\,\d\mu_{\fh_2}(Z), 
 \label{BergmanMetO-Sp}
\end{align}
where $Z=\left[\begin{smallmatrix} z_1 & z_2 \\ z_2 & z_3 \end{smallmatrix}\right]\in \fh_2$ and $\d Z=\prod_{j=1}^{3}2^{-1}|\d z_j\wedge \d \bar z_j|$. 

Since ${\bf GSp}_2(\Z_p)$ stabilizes the lattice $V(\Z_p)\cong \cL_p$, we have the containment ${\bf GSp}_2(\Z_p)\subset \rho^{-1}({\bG}(\Z_p))$, which should be the equality because ${\bf GSp}_2(\Z_p)$ is a maximal compact subgroup of ${\bf GSp}_2(\Q_p)$, i.e., 
$$
\rho({\bf GSp}_2(\Z_p))={\bG}(\Z_p) \quad (p<\infty).
$$
Recall the spherical representations $\pi_p^{\rm ur}(\nu)$ defined in \S~\ref{sec:Intro} and $\pi_p^{\bG}(\nu)$ defined in \S~\ref{sec:Othogonal}; they are related by $\rho$ as expected. 
\begin{lem} For $\nu\in \fX_p$, $\pi_{p}^{\bG}(\nu)\circ \rho\cong \pi_{p}^{\rm ur}(\nu)$.
\end{lem}
\begin{proof} By \eqref{rhoOnSP1} and \eqref{rhoOnSP2}, we see that $\rho({\bf B})={\mathbb B}$ and
$$\rho(\diag(t_1,t_2,\lambda t_1^{-1},\lambda t_2^{-1}))=\diag(a_1,a_2,1,a_1^{-1},a_2^{-1}) 
$$
with $a_1=\lambda^{-1}t_1t_2$ and $a_2=t_1t_2^{-1}$ for $(t_1,t_2,\lambda)\in (\Q_p^{\times})^3$. For $(a_1,a_2)$ and $(t_1,t_2,\lambda)$ related by this equation, it is easy to confirm 
$$
\chi_\nu^{\bG}(\diag(a_1,a_2,1,a_1^{-1},a_2^{-1})=\chi_\nu(\diag(t_1,t_2,\lambda t_1^{-1},\lambda t_2^{-1}))
$$
by \eqref{GspBorelchar} and \eqref{SOBorelchar}. Thus $\chi_{\nu}^{\bG}\circ \rho=\chi_{\nu}$, which implies $I_p^{\bG}(\nu)\circ \rho=I_p(\nu)$ for any $\nu \in \fX_p$. Since $\rho({\bf G}(\Z_p))=\bG(\Z_{p})$, the ${\bf G}(\Z_p)$-spherical constituent $\pi_p^{\rm ur}(\nu)$ of $I_p(\nu)$ and the $\bG(\Z_p)$-spherical constituent $\pi_{p}^{\bG}(\nu)$ of $I_p^{\bG}(\nu)$ corresponds to each other by $\rho$. 
\end{proof}

\begin{prop} \label{SpO-Isom}
 The map ${\rm F} \mapsto \Phi$ defined as
$$
\Phi(Z)={\rm F}(j_{\cD}(Z),1), \quad Z\in \fh_2
$$
yields a linear bijection $j^{*}_{\cD}:S_l(\bK_\fin^*) \rightarrow S_l({\bf Sp}_2(\Z))$ preserving the actions of the Hecke algebras under the isomorphism $\rho^{*}:\cH(\sG(\A_\fin)\sslash \bK_\fin)\rightarrow \cH({\bf G}(\A_\fin)\sslash {\bf G}(\widehat \Z))$. Let ${\rm F} \in S_l(\bK_\fin^*)$ be a Hecke eigenfunction and set $\Phi=j_{\cD}^{*}({\rm F})$; then 
\begin{align*}
L_\fin(s,\pi_\Phi)&=L_\fin(s,{\rm F}), \quad \|\Phi\|^2=16\,\|{\rm F}\|^2.
\end{align*}
Moreover, for any fundamental discriminant $D<0$ and for any character $\chi$ of ${\rm Cl}_D$, we have
$$
 R(\Phi,D,\chi)={\rm R}({\rm F},D,\chi).
$$
\end{prop}
\begin{proof} The relation between $\|{\rm F}\|^2$ and $\|\Phi\|^2$ follows from \eqref{BergmanMetO-Sp}. Here, a care is necessary because $\|{\rm F}\|^2$ is defined by the integral over $\Gamma^{+}(Q)\bsl \cD$ whereas $j_{\cD}$ is bijective only on the double cover $\bG(\Z)\bsl \cD$ of $\Gamma^{+}(Q)\bsl \cD$.
\end{proof}

\end{document}